\documentclass{amsart}

\usepackage{begnac}

\DeclareMathOperator{\SH}{SH}

\DeclareMathOperator*\limF{{\lim}^F}

\begin{document}

\title{Star sorts, Lelek fans, and the reconstruction of non-$\aleph_0$-categorical theories in continuous logic}

\author{Itaï \textsc{Ben Yaacov}}

\address{Itaï \textsc{Ben Yaacov} \\
  Université Claude Bernard -- Lyon 1 \\
  Institut Camille Jordan, CNRS UMR 5208 \\
  43 boulevard du 11 novembre 1918 \\
  69622 Villeurbanne Cedex \\
  France}

\urladdr{\url{http://math.univ-lyon1.fr/~begnac/}}

\thanks{Author supported by ANR project AGRUME (ANR-17-CE40-0026).}


\begin{abstract}
  We prove a reconstruction theorem valid for arbitrary theories in continuous (or classical) logic in a countable language, that is to say that we provide a complete bi-interpretation invariant for such theories, taking the form of an open Polish topological groupoid.

  More explicitly, for every such theory $T$ we construct a groupoid $\bG^*(T)$ that only depends on the bi-interpretation class of $T$, and conversely, we reconstruct from $\bG^*(T)$ a theory that is bi-interpretable with $T$.
  The basis of $\bG^*(T)$ (namely, the set of objects, when viewed as a category) is always homeomorphic to the Lelek fan.

  We break the construction of the invariant into two steps.
  In the second step we construct a groupoid from any sort of codes for models, while in the first step such a sort is constructed.
  This allows us to place our result in a common framework with previously established ones, which only differ by their different choice of sort of codes.
\end{abstract}

\maketitle

\tableofcontents

\section*{Introduction}

This paper deals with what we have come to refer to as \emph{reconstruction theorems}.
By this we mean a procedure that associates to a theory $T$ (possibly under some hypotheses) a topological group-like object that is a complete bi-interpretation invariant for $T$.
In other words, if $T'$ is bi-interpretable with $T$, then we associate to it the same object (up to an appropriate notion of isomorphism), and conversely, the isomorphism class of this object determines the bi-interpretation class of $T$.

The best-known result of this kind is due to Coquand, and appears in Albrandt \& Ziegler \cite{Ahlbrandt-Ziegler:QuasiFinitelyAxiomatisable}.
It states that if $T$ is an $\aleph_0$-categorical theory (in a countable language), then the topological group $G(T) = \Aut(M)$, where $M$ is the unique countable model, is such an invariant.
This was originally proved for theories in classical (Boolean-valued) logic, and subsequently extended by Kaïchouh and the author \cite{BenYaacov-Kaichouh:Reconstruction} to continuous (real-valued) logic.

In \cite{BenYaacov:ReconstructionGroupoid} we proposed a reconstruction result that also covers some non-$\aleph_0$-categorical theories, using a topological groupoid (rather than a group) as invariant.
The result was presented in two times, first for classical logic and then for the more general continuous logic.
This was not done for the sake of presentation (do the more familiar case first), but because of a fundamental difference between the two cases.
In classical logic, we have a straightforward construction of a sort of ``codes of models'' (more about this later).
In continuous logic, on the other hand, no such construction exists in general, and we were reduced to assuming that such a sort (satisfying appropriate axioms) existed, and was given to us.
Worse still, we gave an example of a theory for which no such sort existed, and consequently, for which our reconstruction theorem was inapplicable.

In the present paper we seek to remedy this deficiency, proposing a reconstruction theorem that holds for all theories (in a countable language).
This time, we work exclusively in continuous logic, keeping in mind that this contains classical logic as a special case.

In \autoref{sec:Sorts} we provide a few reminders regarding continuous logic in general, and interpretable sorts in particular.
We (re-)define the notions of interpretation and bi-interpretation, in a manner that is particularly appropriate for the use we shall make of them, and that avoids the rather tedious notions of interpretation schemes.

In \autoref{sec:Coding} we discuss various ways in which one sort $E$ can be ``coded'' in another sort $D$, both uniform (e.g., $E$ is interpretable in $D$) and non-uniform (e.g., each $a \in E$ is in the definable closure of some $b \in D$).
We define a \emph{coding sort} $D$ as a sort which codes models.
Every sort is coded in a coding sort in a non-uniform fashion, and therefore in a uniform fashion as well.

In \autoref{sec:Reconstruction} we associate to a coding sort $D$ a topological groupoid $\bG_D(T)$, from which a theory $T_{2D}$, bi-interpretable with $T$, can be recovered.
In particular, $\bG_D(T)$ determines the bi-interpretation class of $T$.
If, in addition, $D$ only depends on the bi-interpretation class of $T$, then so does $\bG_D(T)$, in which case it is a complete bi-interpretation invariant.
We point out, rather briefly, how previous reconstruction theorems fit in this general setting.

In \autoref{sec:StarSpace} and \autoref{sec:StarSort} we define \emph{star spaces} and \emph{star sorts}.
These, by their very nature, require us to work in continuous (rather than classical) logic.
In particular, we define a notion of a \emph{universal star sort}, and show that if it exists, then it is unique up to definable bijection, and only depends on the bi-interpretation class of $T$.

In \autoref{sec:Witnesses} we use the star sort formalism to give a construction that is analogous to, though not a direct generalisation of, the construction of the coding sort for classical theories in \cite{BenYaacov:ReconstructionGroupoid}.
We then prove that the resulting sort is a universal star sort, so one always exists.
Moreover, the construction is independent of the theory: we simply construct, for any countable language $\cL$, a star sort $D^*$ that is universal in any $\cL$-theory, complete or incomplete.

We conclude in \autoref{sec:UniversalStarSort}, showing that the universal star sort must be a coding sort, whence our most general reconstruction theorem: in a countable language, the groupoid $\bG_{D^*}(T)$ is a complete bi-interpretation invariant for $T$.
We also show that the type-space of the sort $D^*$, relative to any complete theory $T$, is the Lelek fan $L$.
Finally, in case $T$ does fall into one of the cases covered by previous results, we show that our last result can be viewed as some kind of generalisation.
More precisely, using the Lelek fan, we can recover the coding sort $D^*$, and therefore the corresponding groupoid $\bG_{D^*}(T)$, from those given by the earlier results.

\section{Sorts and interpretations}
\label{sec:Sorts}

As said in the introduction, we are going to work exclusively in continuous first order logic, and assume that the reader is familiar with it.
For a general exposition, see \cite{BenYaacov-Usvyatsov:CFO,BenYaacov-Berenstein-Henson-Usvyatsov:NewtonMS}.
We allow formulas to take truth values in arbitrary compact subsets of $\bR$, so connectives are arbitrary continuous functions from $\bR^n$ to $\bR$.
For a countable family of connectives, it will suffice to take all rational constants, addition and multiplication, to which we add the absolute value operation.
Closing these under composition yields a (countable) family of functions that is dense among all continuous functions on each compact subset of $\bR^n$.

\begin{ntn}
  \label{ntn:MaxMinDotMinus}
  Using the absolute value operation we may define maximum and minimum directly (i.e., without passing to a limit).
  We shall use infix notation $\vee$ and $\wedge$ for those.
  We shall also write $t \dotminus s$ for the \emph{truncated subtraction} $(t-s) \vee 0$.
\end{ntn}

We allow the language to be many-sorted.
Some of the time we also require the language to be countable, which means in particular that the set of sorts is countable, although this will not be a requirement for the present section.

We are going to talk quite a bit about sorts and interpretations, so let us begin with a few reminders.
By a \emph{sort} we mean an interpretable sort in the sense of continuous logic, as discussed, for example, in \cite{BenYaacov-Kaichouh:Reconstruction,BenYaacov:ReconstructionGroupoid}.
Sorts are obtained by closing the family of basic sorts (namely, sorts named in the language) by
\begin{itemize}
\item adding the constant sort $\{0,1\}$ (so it is always implicitly interpretable),
\item countable product,
\item quotient by a definable pseudo-distance (in a model that is not saturated, this may also require a passage to the completion), and
\item non-empty definable subset.
\end{itemize}

We follow the convention that natural numbers are coded by sets $n = \{0,\ldots,n-1\} \in \bN$, so $\{0,1\}$ may sometimes be denoted by $2$ (this is especially true of its powers: the Cantor space is $2^\bN$).

Throughout, by \emph{definable} we mean definable by a formula, without parameters (unless parameters are given explicitly).
Any function $\{0,1\} \rightarrow \bR$ is a formula on the sort $\{0,1\}$.
Formulas on a finite product of sorts are constructed in the usual way, using function and predicate symbols, connectives and quantifiers, and closing the lot under uniform limits.
In particular, if $\varphi_i(x)$ are formulas on a sort $D$ for $i < 2^n$, then $\varphi(i,x) = \varphi_i(x)$ is a formula on $2^n \times D$.
Formulas on an infinite product of sorts consist of all formulas on finite sub-products (extended to the whole product through the addition of dummy variables), as well as all uniform limits of such (where the sub-products through which they factor may vary).
If $\overline{d}$ is a definable pseudo-distance on a sort $D$ (defined by a formula on $D \times D$), then formulas on the quotient $(D,\overline{d})$ are formulas on $D$ that are uniformly continuous with respect to $\overline{d}$.
Similarly, for formulas on a product of several quotient sorts.

Finally, we recall that a definable subset of a sort $D$ is a subset $E \subseteq D$, the distance to which is definable (this is significantly more involved than the notion of a definable subset in classical logic).
Equivalently, if for every formula $\varphi(x,y)$, where $x$ is a variable in $D$ and $y$ is a tuple of variables in arbitrary sorts, the predicate $\sup_{x \in E} \varphi(x,y)$ is definable by a formula $\psi(y)$.
Formulas on a product of definable subsets of sorts are restrictions of formulas on the corresponding product of ambient sorts.

Notice that every compact metric space is a quotient space of $2^\bN$ by a continuous pseudo-distance, and therefore a sort, on which the formulas are the continuous functions.
Conversely, we could have chosen any non-trivial compact metric space as a basic constant sort in place of $\{0,1\}$ (the other obvious candidate being $[0,1]$), and realise $\{0,1\}$ as any two-point set therein.

\begin{rmk}
  \label{rmk:PseudoDistanceFromFormula}
  An obvious, yet crucial remark, is that if $\varphi(x,y)$ is an arbitrary formula on $E \times D$, then
  \begin{gather*}
    d_\varphi(y,y') = \sup_{x \in E} \, |\varphi(x,y) - \varphi(x,y')|
  \end{gather*}
  defines a pseudo-distance on $D$.
  In addition, if $D = E$, and $\varphi$ happens to define a pseudo-distance on $D$, then it agrees with $d_\varphi$.

  This has numerous useful consequences, let us state two of them explicitly.
  First of all, one may be bothered by the fact that a formula $\varphi(x,y)$ defining a pseudo-distance on a sort $D$ may depend on the structure(s) under consideration.
  However, we may restrict the ``quotient by a pseudo-distance'' step to pseudo-distances of the form $d_\varphi$, that always define pseudo-distances, without any loss of generality.

  A second consequence is that if $E \subseteq D$ are two sorts, then every definable pseudo-distance $d$ on $E$ extends to one on $D$.
  Indeed, extend it first in an arbitrary fashion to a formula $\varphi(x,y)$ on $E \times D$.
  Then $d_\varphi$ is a pseudo-distance on $D$, and it agrees with $d$ on $E$.
\end{rmk}

\begin{rmk}
  \label{rmk:DependenceOnTheory}
  A formula $\psi(x)$ defining the distance to a subset is another property that depends on the structure under consideration, or on its theory.
  However, we do not know a general construction of definable sets from arbitrary formulas, analogous to that of \autoref{rmk:PseudoDistanceFromFormula}, and have good reason to believe that none such exists.

  In other words, as far as we know, the set of interpretable sorts depends in a non-trivial way on the theory.
  This makes it all the more noteworthy that our construction of the universal star sort as $D^*_\Phi$ can be carried out in a manner that depends only on the language, and not on the theory.
\end{rmk}

A \emph{definable map} between two sorts $\sigma\colon D \rightarrow E$ is one whose graph is the zero-set of some formula.
Composing a formula with a definable map yields another formula.
A special case of such a composition is the formula $d\bigl( \sigma(x), y\bigr)$, on the product $D \times E$, whose zero-set is indeed the graph of $\sigma$.
Every formula is uniformly continuous in its arguments, and $d\bigl(\sigma(x), y \bigr)$ is no exception.
It follows that every definable map $\sigma\colon D \rightarrow E$ is uniformly continuous.

Two sorts that admit a definable bijection are, for most intents and purposes (in particular, for those of the present paper) one and the same.
Moreover, every sort is in definable bijection with one obtained from the basic sorts by applying each of the operations once, in the given order, so we may pretend that every sort is indeed of this form.
Similarly, we may say that a sort $D$ (which may be a basic sort, or one that has already been obtained through some interpretation procedure) is \emph{interpretable} in a family of sorts $(E_i)$ if we can construct from this family $(E_i)$ a sort $D'$ that admits a definable bijection with $D$.

Consider two languages $\cL \subseteq \cL'$, where $\cL'$ is allowed to add not only symbols, but also sorts.
If $M'$ is an $\cL'$-structure, and $M$ is the $\cL$-structure obtained by dropping the sorts and symbols not present in $\cL$, then $M$ is \emph{the $\cL$-reduct} of $M'$ and $M'$ is \emph{an $\cL'$-expansion} of $M$.
If $T'$ is an $\cL'$-theory and $T$ is the collection of $\cL$-sentences in $T'$, then $T$ is also the theory of all $\cL$-reducts of models of $T'$ (notice, however, that an arbitrary model of $T$ need only admit an elementary extension that is a reduct of a model of $T'$).
In this situation we say that $T$ is \emph{the $\cL$-reduct} of $T'$ and that $T'$ is \emph{an $\cL'$-expansion} of $T$.

One special case of an expansion is a \emph{definitional expansion}, in which $\cL$ and $\cL'$ have the same sorts, and each new symbol of $\cL'$ admits an $\cL$-definition in $T'$.
In this case, $T'$ is entirely determined by $T$ together with these definitions.
A more general case is that of an \emph{interpretational expansion} of $T$, where $T'$ identifies each new sort of $\cL'$ with an interpretable sort of $T$, and gives $\cL$-definitions to all new symbols in $\cL'$ (for this to work we also require $\cL'$ to contain, in particular, those new symbols that allow $T'$ to identify the new sorts with the corresponding interpretable ones).
Again, $T$, together with the list of interpretations of the new sorts and definitions of the new symbols, determine $T'$.
Moreover, unlike the general situation described in the previous paragraph, here every model of $T$ expands to a model of $T'$.

\begin{dfn}
  \label{dfn:Interpretation}
  Let $T$ and $T'$ be two theories, say in disjoint languages.
  We say that $T'$ is \emph{interpretable} in $T$ if $T'$ is a reduct of an interpretational expansion of $T$.
  The two theories are \emph{bi-interpretable} if they admit a common interpretational expansion (which is stronger than just each being interpretable in the other).
\end{dfn}

A theory has the same sorts (up to a natural identification) as an interpretational expansions.
Therefore, somewhat informally, we may say that two theories are bi-interpretable if and only if they have the same sorts.

Let us consider a few more possible constructions of sorts that will become useful at later stages, and show that they can be reduced to the basic construction steps that we allow.

\begin{lem}
  \label{lem:SortInverseLimit}
  Let
  \begin{gather*}
    D_0 \stackrel{\pi_0}\twoheadleftarrow D_1 \stackrel{\pi_1}\twoheadleftarrow \cdots
  \end{gather*}
  be an inverse system of sorts with surjective definable maps $\pi_n\colon D_{n+1} \twoheadrightarrow D_n$.
  Then the inverse limit $D = \varprojlim D_n \subseteq \prod D_n$ is again a sort, which we may equip with the distance
  \begin{gather}
    \label{eq:SortInverseLimitDistance}
    d(x,y) = \sum_n \, \Bigl(2^{-n} \wedge d(x_n, y_n)\Bigr)
  \end{gather}
  (or with the restriction of any other definable distance on $\prod D_n$).
\end{lem}
\begin{proof}
  Indeed, $D$ is the zero-set in $\prod D_n$ of the formula
  \begin{gather*}
    \varphi(x) = \sum_n \, \Bigl(2^{-n} \wedge d\bigl( x_n, \pi_n(x_{n+1}) \bigr)\Bigr).
  \end{gather*}
  Let $\varepsilon > 0$, and choose $N \in \bN$ large enough depending on $\varepsilon$, and $\delta > 0$ small enough depending on both.
  Let $a \in \prod D_n$, and assume that $\varphi(a) < \delta$.
  Since the maps are surjective, there exists $b \in D$ such that $b_N = a_N$.
  This determines $b_n$ for all $n \leq N$, and having chosen $\delta$ small enough, we have $d(a_n,b_n)$ as small as desired for all $n \leq N$.
  Having chosen $N$ large enough, this yields $d(a,D) \leq d(a,b) < \varepsilon$.

  In other words, we have found a formula $\varphi(x)$ that vanishes on $D$, such that satisfies $\varphi(x) < \delta = \delta(\varepsilon)$ implies $\varphi(x,D) < \varepsilon$.
  This implies that $D$ is a definable subset (see \cite{BenYaacov-Berenstein-Henson-Usvyatsov:NewtonMS}).
\end{proof}

\begin{prp}
  \label{prp:IncreasingLimit}
  Assume that $(D_n)$ is a sequence of sorts, equipped with isometric definable embeddings $D_n \hookrightarrow D_{n+1}$.
  For convenience, let us pretend these embeddings are the identity map, so $D_0 \subseteq D_1 \subseteq \cdots \subseteq D_n \subseteq \cdots$ is a chain.
  Assume moreover that the sequence is Cauchy in the Hausdorff distance.
  In other words, assume that if $n$ is large enough and $n \leq m$, then
  \begin{gather*}
    d^H(D_n,D_m) = \sup_{x \in D_m} \, \inf_{y \in D_n} \, d(x,y)
  \end{gather*}
  is as small as desired.

  Then the completion $E = \widehat{\bigcup D_k}$ is a sort (with definable isometric embedding $D_n \subseteq E$).
  If $\varphi(x,y)$ is a formula on $E \times F$, for some sort (or product of sorts) $F$, and $\varphi_n$ is its restriction to $D_n \times F$, then $(\varphi_n)$ is an equicontinuous compatible family (by compatible, we mean that each $\varphi_n$ is the restriction of $\varphi_{n+1}$).
  Conversely, every such family arises from a unique formula on $E \times F$.
\end{prp}
\begin{proof}
  Assume first that we have a large ambient sort $E_1$ and compatible isometric embeddings $D_n \subseteq E_1$.
  Since each $D_n$ is a sort, the distance $d(x,D_n) = \inf_{y \in D_n} \, d(x,y)$ is definable in $E_1$.
  By hypothesis, these formulas converge uniformly, and their limit is $d(x,E)$.
  Then $E$ is a definable subset of $E_1$, and therefore a sort.

  In the general case, we are going to construct $E_1$ as a quotient of $E_0 = \prod D_n$, whose members we may view as sequences in $E$.
  We may freely pass to a sub-sequence, and assume that $d^H(D_n,D_{n+1}) < 2^{-n-1}$.
  Say that $a \in E_0$ \emph{converges quickly} if $d(a_n,a_m) \leq 2^{-n} + 2^{-m}$, or equivalently, if $d(a_n,b) \leq 2^{-n}$ where $a_n \rightarrow b$ in $E$.
  By our hypothesis regarding the rate of convergence of $(D_n)$, every $b \in E$ is the limit of a quickly converging sequence.

  Recall the \emph{forced limit} construction from \cite{BenYaacov-Usvyatsov:CFO}.
  Formally, it consists of a continuous function $\limF\colon \bR^\bN \rightarrow \bR$, which is monotone, $1$-Lipschitz in the supremum norm on $\bR^\bN$, and most importantly, if $t_n \rightarrow s$ fast enough, say $|t_n - s| \leq 2^{-n}$, then $\limF (t_n : n \in \bN) = s$.
  We render the expression $\limF (t_n : n \in \bN)$ as $\limF_{n \rightarrow \infty} t_n$, considering it a limit construct.
  Since $\limF$ is continuous, we may apply it to formulas.

  Let us fix $n$, and define on $D_n \times E_0$ a formula
  \begin{gather*}
    \rho_n(x,y) = \limF_{m \rightarrow \infty} \, d(x,y_m).
  \end{gather*}
  If $b \in E_0$ converges quickly to $c \in E$, then $\rho_n(a,b) = d(a, c)$ for every $a \in D_n$.
  When $b \in E_0$ does not converge quickly (or possibly, at all), the value $\rho_n(a,b)$ is well defined, but potentially meaningless.
  If $n \leq k$, then $\rho_n$ is the restriction of $\rho_k$, so we may just denote all of them by $\rho$.

  As in \autoref{rmk:PseudoDistanceFromFormula}, we define pseudo-distances on $E_0$ by
  \begin{gather*}
    d_{\rho_n}(y,y') = \sup_{x \in D_n} \, |\rho(x,y) - \rho(x,y')|.
  \end{gather*}
  The sequence of formulas $(d_{\rho_n})$ is increasing.
  Moreover, if $x,y \in D_n$ and $z \in E_0$, then
  \begin{gather*}
    |\rho(x,z) - \rho(y,z)| \leq \sup_m \, |d(x,z_m) - d(y,z_m)| \leq d(x,y),
  \end{gather*}
  so $d_{\rho_n} \leq d_{\rho_{n+1}} \leq d_{\rho_n} + 2^{-n}$.
  Therefore the sequence $(d_{\rho_n})$ converges uniformly to a formula $d_\rho$ on $E_0 \times E_0$, which must define a pseudo-distance as well.
  Let $E_1 = (E_0,d_\rho)$ be the quotient sort.
  By definition, each $\rho_n(x,y)$ is $1$-Lipschitz in $y$ with respect to $d_\rho$, so it may be viewed as a formula on $D_n \times E_1$.
  It is also $1$-Lipschitz in $x$ with respect to the distance on $D_n$.

  Consider $a \in D_k$ and $b,c \in E_1$, and assume that $b_n \rightarrow a$ quickly (but $c$ may be quite arbitrary).
  We have already observed that $\rho(x,b) = d(x,a)$ for every $x \in D_n$, for every $n$.
  Then, for every $n \geq k$:
  \begin{gather*}
    d_{\rho_n}(b,c)
    = \sup_{x \in D_n} \, |\rho(x,b) - \rho(x,c)|
    = \sup_{x \in D_n} \, |d(x,a) - \rho(x,c)|
    = \rho(a,c),
  \end{gather*}
  so $d_\rho(b,c) = \rho(a,c) = \rho_k(a,c)$.
  If follows that the class of $b$ in $E_1$ only depends on $a$.
  Moreover, the map $\sigma_k\colon D_k \rightarrow E_1$, that sends $a$ to the class of any $b \in E_0$ that converges quickly to $a$, is definable, by $d_\rho\bigl( \sigma_k(x), y\bigr) = \rho_k(x,y)$.

  If $b,b'\in E_0$ both converge quickly to $a,a' \in D_k$, respectively, then the same reasoning as above yields $d_{\rho_n}(b,b') = d(a,a')$ for every $n \geq k$, and therefore $d_\rho(b,b') = d(a,a')$.
  Therefore, $\sigma_k\colon D_k \rightarrow E_1$ is an isometric embedding for each $k$.
  Since the $\rho_k$ are restrictions of one another, these embeddings are compatible, and we have successfully reduced to the special case treated in the beginning of the proof.

  Regarding formulas, the only thing we need to prove is that any compatible equicontinuous family of formulas $\varphi_n(x,y)$ on $D_n \times F$, is the restriction of a formula on $E \times F$.
  Notice that our hypotheses imply that the formula $\varphi_n$ are uniformly bounded, say $|\varphi_n| \leq M$.
  We may now construct an inverse modulus of continuity, namely a continuous function $\Delta^{-1}\colon (0,\infty) \rightarrow (0,\infty)$ such that $|\varphi_n(x,y) - \varphi_n(x',y)| \leq \Delta^{-1} \circ d(x,x')$ (see \cite{BenYaacov-Usvyatsov:CFO}; since the family is equicontinuous, we can do this simultaneously for all $\varphi_n$).
  Define on $E \times F$ formulas
  \begin{gather*}
    \psi_n(x,y) = \inf_{x' \in D_n} \bigl( \varphi_n(x',y) + \Delta^{-1} \circ d(x,x') \bigr).
  \end{gather*}
  Then $\psi_n$ agrees with $\varphi_n$ on $D_n \times F$, and equicontinuity together with the convergence of $(D_n)$ in $d^H$ implies that $(\psi_n)$ converge uniformly to a formula $\psi(x,y)$ on $E \times F$, that must extend each $\varphi_n$, as claimed.
\end{proof}

It was pointed out by James Hanson that our \autoref{prp:IncreasingLimit} already appeared in his Ph.D.\ thesis \cite[Proposition~3.4.8]{Hanson:PhD}.
Similarly, in \cite[Remark~3.5.7]{Hanson:PhD} he asserts (without proof) something that, to the extent that we understand it (terminology and notation being somewhat non-standard), is related to our \autoref{prp:NonUniformCoding}.

\section{Coding sorts in other sorts}
\label{sec:Coding}

If $a$ and $b$ are two elements in sorts $E$ and $D$ in some structure (model of $T$), then $a$ is \emph{definable} from $b$, or lies in the \emph{definable closure} of $b$, in symbols $a \in \dcl(b)$, if $a$ is the unique realisation of $\tp(a/b)$ in that structure, as well as in any elementary extension.
This implies, and indeed, is equivalent to, the predicate $d(x,a)$ being definable with $b$ as parameter, say by a formula $\varphi(x,b)$ (see \cite{BenYaacov:DefinabilityOfGroups}).

Let us consider two sorts $D$ and $E$.
In what sense(s) can $E$ be coded in $D$?
A fairly \emph{uniform} fashion for this to happen is if $E$ is interpretable in $D$, i.e., if it embeds definably in a quotient of $D^\bN$, or, at the very worst, $D^\bN \times 2^\bN$.
This would imply a \emph{non-uniform} version: for every $a \in E$ there exists $b \in D^\bN$ such that that $a \in \dcl(b)$.
In fact, the converse implication holds as well -- this follows fairly easily from \autoref{prp:NonUniformCoding} below, together with the presentation of $\widehat{\bigcup D_n}$ as a subset of a quotient of $\prod D_n$.

In any case, we want to explore a stronger condition of ``non-uniform coding'', by singletons in $D$.

\begin{prp}
  \label{prp:NonUniformCoding}
  Let $E$ and $D$ be sorts of a theory $T$.
  Assume that for every $a \in E$ (in a model of $T$) there exists $b \in D$ (possibly in an elementary extension) such that $a \in \dcl(b)$.
  Then $E$ can be embedded in a limit sort of the form $\widehat{\bigcup D_n}$, as per \autoref{prp:IncreasingLimit}, where each $D_n$ is a quotient of $D \times 2^\bN$.
\end{prp}
\begin{proof}
  Consider a type $p \in \tS_E(T)$, so $p = \tp(a)$ for some $a \in E$ in a model of $T$.
  We may assume that $b \in D$ in the same model is such that $a \in \dcl(b)$, as witnessed by $d(x,a) = \varphi_p(x,b)$.

  Let (with $\varepsilon > 0$)
  \begin{gather*}
    \psi_p(x,y) = \sup_{x'} \, |d(x,x') - \varphi_p(x',y)|,
    \\
    \chi_{p,\varepsilon}(y) = 1 \dotminus \bigl( \inf_x \, \psi_p(x,y)/\varepsilon \dotminus 1 \bigr).
  \end{gather*}
  The formula $\psi_p(x,y)$ measures the extent to which $\varphi_p(x',y)$ fails to give us the distance to $x$.
  The formula $\chi_{p,\varepsilon}(y)$ tells us whether $x' \mapsto \varphi_p(x',y)$ is close to being the distance to \emph{some} $x \in E$: $\chi_{p,\varepsilon}(y) = 1$ if $y$ codes some $x$ quite well (error less than $\varepsilon$), vanishes if $y$ does not code anything well enough (error at least $2\varepsilon$), and in all cases its value lies in $[0,1]$.
  Of course, $\psi_p(a,b) = 0$, so $\inf_y \psi_p(x,y) < \varepsilon$ defines an open neighbourhood of $p$.

  Let us fix $\varepsilon > 0$ and let $p$ vary.
  Then the conditions $\inf_y \psi_p(x,y) < \varepsilon$ define an open covering of $\tS_E(T)$.
  By compactness, there exists a family $(p_i : i < n)$ such that for every $q \in \tS_E(T)$, $\inf_y \psi_{p_i}(q,y) < \varepsilon$ for at least one $i < n$.
  Repeating this, with smaller and smaller $\varepsilon$, we may construct a sequence of types $(p_n)$, as well as $\varepsilon_n \rightarrow 0$ such that for every $n_0$, the open conditions $\inf_y \psi_{p_n}(x,y) < \varepsilon_n$ for $n \geq n_0$ cover $\tS_E(T)$.

  Let $n \in \bN$.
  We may view $n = \{0, \ldots, n-1\}$ as a quotient of $2^\bN$, and similarly for $[0,1]$.
  Therefore, $D \times n \times [0,1]$ is a quotient of $D \times 2^\bN$.
  For $(x,y,k,t) \in E \times D \times n \times [0,1]$, define
  \begin{gather*}
    \rho_n(x,y,k,t) = t \cdot \chi_{p_k,\varepsilon_k}(y) \cdot \varphi_{p_k}(x,y).
  \end{gather*}
  This is indeed a formula, giving rise to a pseudo-distance on $D \times n \times [0,1]$:
  \begin{gather*}
    d_{\rho_n}(y,k,t,y',k',t') = \sup_{x \in E} \, | \rho_n(x,y,k,t) - \rho_n(x,y',k',t')|.
  \end{gather*}
  In fact, we may drop $n$ and just write $\rho$ and $d_\rho$: the only role played by $n$ is being greater than $k$.

  Let $D_n$ be the quotient $\bigl( D \times n \times [0,1], d_\rho \bigr)$ (which is, in turn, a quotient of $D \times 2^\bN$).
  The inclusion $D \times n \times [0,1] \subseteq D \times (n+1) \times [0,1]$ induces an isometric embedding $D_n \hookrightarrow D_{n+1}$.
  Therefore, in order to show that the hypotheses of \autoref{prp:IncreasingLimit} are satisfied, all we need to show is that for $n \leq m$ large enough, every member of $D_m$ is close to some member of $D_n$.

  Let $\varepsilon > 0$ be given.
  Find $n_0$ such that $\varepsilon_n < \varepsilon$ for $n \geq n_0$.
  Then, by compactness, find $n_1 > n_0$ such that $\inf_y \psi_{p_n}(x,y) < \varepsilon_n$ for $n_0 \leq n < n_1$ cover $\tS_E(T)$.
  Assume now that $n_1 \leq m$, and let $[b,k,t]$ be some class in $D_m$.
  If $k < n_1$, then $[b,k,t] \in D_{n_1}$.
  If $\inf_x \psi_{p_k}(x,b) \geq 2\varepsilon_k$, then $\rho_n(x,b,k,t) = 0$ regardless of $x$, so $[b,k,t] = [b,0,0] \in D_{n_1}$.
  We may therefore assume that $n_1 \leq k < m$ and there exists $a \in E$ such that $\psi_{p_k}(a,b) < 2\varepsilon_k$.

  By our hypothesis regarding the covering of $\tS_E(T)$, there exists $n_0 \leq \ell < n_1$ such that $\inf_y \psi_{p_\ell}(a,y) < \varepsilon_\ell$.
  Let $c \in D$ be such that $\psi_{p_\ell}(a,c) < \varepsilon_\ell$, and let $s = t \cdot \chi_{p_k,\varepsilon_k}(b)$.
  Then
  \begin{gather*}
    \inf_x \, \psi_{p_\ell}(x,c) < \varepsilon_\ell,
    \qquad
    \chi_{p_\ell,\varepsilon_\ell}(c) = 1,
    \qquad
    \rho(x,c,\ell,s) = s \cdot \varphi_{p_\ell}(x,c),
  \end{gather*}
  so
  \begin{align*}
    d_\rho(b,k,t,c,\ell,s)
    & = s \cdot \sup_x \, \bigl| \varphi_{p_k}(x,b) - \varphi_{p_\ell}(x,c) \bigr|
    \\
    & \leq \sup_x \, \bigl| \varphi_{p_k}(x,b) - d(x,a) \bigr| + \sup_x \, \bigl| d(x,a) - \varphi_{p_\ell}(x,c) \bigr|
    \\
    & = \psi_{p_k}(a,b) + \psi_{p_\ell}(a,c) < 2\varepsilon_k + \varepsilon_\ell < 3\varepsilon.
  \end{align*}
  Then $[c,\ell,s] \in D_{n_1}$ is close enough to $[b,k,t]$.
  By \autoref{prp:IncreasingLimit}, a limit sort $F = \widehat{\bigcup D_n}$ exists.

  Now let us embed $E \hookrightarrow F$.
  We have already constructed a family $(\rho_n)$ of formulas on $E \times D_n$, let us write them as $\rho_n(x,z)$.
  Each is $1$-Lipschitz in $z$ by definition of the distance on $D_n$, and they are compatible, so they extend to a formula $\rho(x,z)$ on $E \times F$.

  Consider $a \in E$, and let $\varepsilon > 0$.
  As above, there exists $\ell$ such that $\varepsilon_\ell < \varepsilon$, and $c \in D$ such that $\psi_{p_\ell}(a,c) < \varepsilon_\ell$.
  Let $a' = [c,\ell,1] \in D_{\ell+1} \subseteq F$.
  Again, as above, $\chi_{p_\ell,\varepsilon_\ell}(c) = 1$, so $\rho(x,a') = \varphi_{p_\ell}(x,c)$, and
  \begin{gather*}
    \sup_x \, |d(x,a) - \rho(x,a')|
    = \sup_x \, |d(x,a) - \varphi_{p_\ell}(x,c)|
    = \psi_{p_\ell}(a,c) < \varepsilon_\ell < \varepsilon.
  \end{gather*}
  Doing this with $\varepsilon \rightarrow 0$ we obtain a sequence $(a_n)$ in $F$ such that $\rho(x,a_n)$ converges uniformly to $d(x,a)$.
  By definition of the distance on $F$ as $d_\rho$, this sequence is Cauchy, with limit $\tilde{a} \in F$, say, and $\rho(x,\tilde{a}) = d(x,a)$.
  In particular, for $z \in F$,
  \begin{gather*}
    d(z,\tilde{a}) = \sup_x | \rho(x,z) - \rho(x,\tilde{a})| = \sup_x \, |\rho(x,z) - d(x,a)|,
  \end{gather*}
  so $a \mapsto \tilde{a}$ is definable.
  By the same reasoning, if $a,a' \in E$, then
  \begin{gather*}
    d(\tilde{a}, \tilde{a}') = \sup_x | \rho(x,\tilde{a}) - \rho(x,\tilde{a}')| = \sup_x \, |d(x,a) - d(x,a')| = d(a,a'),
  \end{gather*}
  so the embedding is isometric, completing the proof.
\end{proof}

\begin{rmk}
  \label{rmk:NonUniformCoding}
  A closer inspection of the proof can yield a necessary and sufficient condition (but we shall not use this):
  A sort $E$ can be embedded in a limit sort of the form $\widehat{\bigcup D_n}$, where each $D_n$ is a quotient of $D \times 2^\bN$, if and only if, for every $a \in E$ and $\varepsilon > 0$, there exists $b \in D$ and a formula $\varphi(x,b)$ that approximates $d(x,a)$ with error at most $\varepsilon$.
\end{rmk}

In \autoref{prp:NonUniformCoding}, we cannot replace $D \times 2^\bN$ with just $D$ (if $D$ is a singleton, then any increasing union of quotients of $D$ is a singleton, and yet $E = \{0,1\}$ satisfies the hypothesis of \autoref{prp:NonUniformCoding}).
Instead, let us prove that this does not change much, in the sense that formulas on $D \times 2^\bN$ or on just $D$ are almost the same thing.

\begin{lem}
  \label{lem:ConstantSortFormulas}
  Let $D$ and $E$ be sorts, and let $\varphi(x,t,y)$ be a formula on $D \times 2^\bN \times E$.
  Then $\varphi$ can be expressed as a uniform limit of continuous combinations of formulas on $D \times E$ and on $2^\bN$ separately (where we recall that formulas on $2^\bN$ are just continuous functions $2^\bN \rightarrow \bR$).
\end{lem}
\begin{proof}
  For $n \in \bN$ and $k \in 2^n$, let $\delta_{n,k}(t) = 1$ if $t$ extends $k$, and $0$ otherwise.
  Let also $\tilde{k} \in 2^\bN$ be the extension of $k$ by zeros, and $\varphi_{n,k}(x,y) = \varphi(x,\tilde{k},y)$.

  Then $\varphi_{n,k}$ is a formula on $D \times E$ and $\delta_{n,k}$ is a formula on $2^\bN$, so we may define a formula
  \begin{gather*}
    \varphi_n(x,t,y) = \sum_{k \in \{0,1\}^n} \delta_{n,k}(t) \varphi_{n,k}(x,y).
  \end{gather*}
  Since $\varphi(x,t,y)$ is uniformly continuous in $t$, $\varphi_n \rightarrow \varphi$ uniformly.
\end{proof}

\begin{dfn}
  \label{dfn:CodingSort}
  Let $T$ be a theory, $D$ a sort, and $D^0 \subseteq D$ a definable subset (or even type-definable, namely, the zero-set of a formula).
  We say that $D$ is a \emph{coding sort}, with \emph{exceptional set} $D^0$, if the following holds:
  \begin{enumerate}
  \item
    \label{item:CodingSortModel}
    \emph{Coding models}: if $M \vDash T$ and $a \in D(M) \setminus D^0(M)$, then there exists $N \preceq M$ such that $\dcl(a) = \dcl(N)$.
    We then say that $a$ \emph{codes} $N$.
  \item
    \label{item:CodingSortDensity}
    \emph{Density}: if $M \vDash T$ is separable, then the set of $a \in D(M) \setminus D^0(M)$ that code $M$ is dense in $D(M)$.
  \end{enumerate}
  We may denote a coding sort by $D$ alone, considering $D^0$ as implicitly given together with $D$.
\end{dfn}

The need for an exceptional set will arise at a later stage -- for the time being, we are simply going to ensure that its presence does not cause any trouble.

\begin{dfn}
  \label{dfn:LT2D}
  Let $T$ be a theory, say in a language $\cL$, and let $D$ be a coding sort for $T$.

  We define a single-sorted language $\cL_{2D}$ to consist of a binary predicate symbol for each formula on $D \times D$ (possibly restricting this to a dense family of such formulas).
  We define $T_{2D}$ as the $\cL_{2D}$-theory of $D$ -- namely, the theory of all $D(M)$, viewed naturally as an $\cL_{2D}$-structure, where $M$ varies over models of $T$.
\end{dfn}

Clearly, $T_{2D}$ is interpretable from $T$.
The $2$ is there to remind us that only binary predicates on $D$ are named in the language.

Our aim, in the end, is to recover from a groupoid the theory of some coding sort $D$, and show that is bi-interpretable with $T$.
In particular we need to recover the definable predicates on $D$ from the groupoid.
In \cite{BenYaacov:ReconstructionGroupoid} we managed to recover predicates of all arities, at the price of some additional work.
In the present paper we choose to follow a different path, recovering only binary predicates (i.e., only $T_{2D}$), and instead show that these suffice.

\begin{prp}
  \label{prp:CodingSortBiInterpretation}
  Let $T$ be a theory, say in a language $\cL$, and let $D$ be a coding sort for $T$.
  Then $T_{2D}$ is bi-interpretable with $T$.
\end{prp}
\begin{proof}
  Consider $T'$, obtained from $T$ by adjoining $D$ as a new sort, and naming the full induced structure.
  It is, by definition, an interpretational expansion of $T$, and it will suffice to show that it is also an interpretational expansion of $T_{2D}$.

  By \autoref{lem:ConstantSortFormulas}, every formula on $\bigl( D \times 2^\bN \bigr) \times \bigl( D \times 2^\bN \bigr)$ is definable in $T_{2D}$.
  In particular, every quotient of $D \times 2^\bN$ is interpretable in $T_{2D}$, as is every embedding of one such quotient in another.
  Therefore, if $(D_n)$ is an increasing chain of quotients of $D \times 2^\bN$, that converges in the sense of \autoref{prp:IncreasingLimit}, then $E = \widehat{\bigcup D_n}$ is interpretable in $T_{2D}$.

  Consider now a sort $E$ of $T$.
  Every member of $E$ belongs to a separable model of $T$ and is therefore definable from a member of $D$.
  By \autoref{prp:NonUniformCoding}, we may embed $E$ in a sort $\tilde{E}$ which is of the form $\widehat{\bigcup D_n}$, for appropriate quotients of $D \times 2^\bN$, as in the previous paragraph.
  This presentation of $E$ need not be unique, so let us just fix one such.

  Say $E'$ is another sort of $T$, so $E' \subseteq \tilde{E}' = \widehat{\bigcup D_n'}$ as above.
  Any formula on $\tilde{E} \times \tilde{E}'$ is, by \autoref{prp:IncreasingLimit}, coded by a sequence of formulas on $D_n \times D_n'$ (its restrictions), i.e., by formulas on $\bigl(D \times 2^\bN\bigr)^2$.
  It is therefore definable in $T_{2D}$.
  In particular, the distance to (the copy of) $E$ in $\tilde{E}$ is definable in $T_{2D}$, so each sort $E$ of $T$ can be interpreted in $T_{2D}$ (or at least, some isometric copy of $E$ is interpretable).
  Similarly, every formula on $E \times E'$, can be extended to a formula on $\tilde{E} \times \tilde{E}'$, so it is definable in $T_{2D}$ (on the copies of $E$ and $E'$).

  Consider now a finite product $E = \prod_{i<n} E_i$ of sorts of $T$.
  We have already chosen embeddings $E \subseteq \tilde{E}$ and $E_i \subseteq \tilde{E}_i$ as above.
  The projection map $\pi_i\colon E \rightarrow E_i$ can be coded by a formula on $E \times E_i$, namely
  \begin{gather*}
    \Gamma_{\pi_i}(x,y) = d_{E_i}(x_i,y),
  \end{gather*}
  where $\Gamma$ stands for ``Graph''.
  We have already observed that such a formula is definable in $T_{2D}$.
  It follows that the structure of $E$ as a product of the $E_i$ is definable in $T_{2D}$.
  Finally, any formula on $E_0 \times \cdots \times E_{n-1}$ can be viewed as a unary formula on the product $E$, which is, again, definable in $T_{2D}$.

  In conclusion, we can interpret every sort of $T$ in $T_{2D}$, and recover the full structure on these sorts.
  In other words, $T'$ is indeed an interpretational expansion of $T_{2D}$, completing the proof.
\end{proof}

\section{Groupoid constructions and reconstruction strategies}
\label{sec:Reconstruction}

In this section we propose a general framework for ``reconstruction theorems''.
To any coding sort $D$ (see \autoref{dfn:CodingSort}) we associate a topological groupoid $\bG_D(T)$ from which the theory $T_{2D}$ of \autoref{prp:CodingSortBiInterpretation} can be reconstructed.
Since $T$ is bi-interpretable with $T_{2D}$, the groupoid $\bG_D(T)$ determines the bi-interpretation class of $T$.
If the coding sort is moreover determined by the bi-interpretation class of $T$ (up to definable bijection), then the groupoid is a bi-interpretation invariant.
Various previously known constructions fit in this framework, as well as the one towards which aims the present paper.

For a general treatment of topological groupoids, we refer the reader to Mackenzie \cite{Mackenzie:LieGroupoids}, or, for the bare essentials we shall need here, to \cite{BenYaacov:ReconstructionGroupoid}.
We recall that a \emph{groupoid} $\bG$ is defined either as a small category in which all morphisms are invertible, or algebraically, as a single set (of all morphisms), equipped with a partial composition law and a total inversion map, satisfying appropriate axioms.
When viewed as a category, the set of objects can be identified with the set of identity morphisms, and we call it the \emph{basis} $\bB$ of $\bG$.
In the algebraic formalism, which we follow here, the basis is $\bB = \{e \in \bG : e^2 = e\} \subseteq \bG$.
If $g \in \bG$, then $s(g) = g^{-1} g$ and $t(g) = g g^{-1}$ are both defined, and belong to $\bB$, being the \emph{source} and \emph{target} of $g$, respectively.
The domain of the composition law is
\begin{gather*}
  \dom(\cdot) = \bigl\{ (g,h)  : s(g) = t(h) \bigr\} \subseteq \bG^2.
\end{gather*}

A \emph{topological groupoid} is a groupoid equipped with a topology in which the partial composition law and total inversion map are continuous.
In a topological groupoid the source and target maps $s,t\colon \bG \rightarrow \bB$ are continuous as well, $\bB$ is closed in $\bG$, and $\dom(\cdot)$ closed in $\bG^2$.
A topological groupoid $\bG$ is \emph{open} if, in addition, the composition law $\cdot\colon \dom(\cdot) \rightarrow \bG$ is open, or equivalently, if the source map $s\colon \bG \rightarrow \bB$ (or target map $t\colon \bG \rightarrow \bB$) is open.

A (topological) group is a (topological) groupoid whose basis is a singleton.
Such a topological groupoid is always open.

\begin{dfn}
  \label{dfn:CodingSortGroupoid}
  Let $T$ be a theory in a countable language, and $D$ a coding sort.
  We let $\tS_{D \times D}(T)$ denote the space of types of pairs of elements of $D$.
  We define the following two subsets of $\tS_{D \times D}(T)$:
  \begin{gather*}
    \bG_D^0(T) = \bigl\{ \tp(a,a) : a \in D^0 \bigr\},
    \\
    \bG_D(T) = \bG_D^0(T) \cup \bigl\{ \tp(a,b) : a, b \in D \setminus D^0 \ \& \ \dcl(a) = \dcl(b) \bigr\},
  \end{gather*}
  where $a$ and $b$ vary over all members of $D$ (or $D^0$) in models of $T$.
  We equip $\bG_D(T)$ with the induced topology, as well as with the following inversion law and  partial composition law:
  \begin{gather*}
    \tp(a,b)^{-1} = \tp(b,a),
    \qquad
    \tp(a,b) \cdot \tp(b,c) = \tp(a,c).
  \end{gather*}
  We also write $\bB_D(T)$ for $\tS_D(T)$, and identify $\tp(a) \in \bB_D(T)$ with $\tp(a,a) \in \bG_D(T)$.
  This identifies $\bB_D^0(T) = \tS_{D^0}(T)$ with $\bG_D^0(T)$.
\end{dfn}

Notice that the density hypothesis in \autoref{dfn:CodingSort} implies that $\bG_D(T)$ is dense in $\tS_{D \times D}(T)$.

\begin{conv}
  \label{conv:CodingSortGroupoid}
  We usually consider the theory $T$ and the coding sort $D$ to be fixed and drop them from notation, so $\bG = \bG_D(T)$, $\bB = \bB_D(T)$, and so on.
\end{conv}

\begin{lem}
  \label{lem:CodingSortGroupoid}
  Let $D$ be a coding sort for $T$.
  \begin{enumerate}
  \item As defined above $\bG = \bG_D(T)$ is a Polish open topological groupoid with basis $\bB = \bB_D(T)$.
  \item If $g = \tp(a,b) \in \bG$, then $s(g) = \tp(b) \in \bB$ is its source, and $t(g) = \tp(a) \in \bB$ its target.
  \item
    \label{item:CodingSortGroupoidBasisNeighbourhoods}
    If $d$ is a definable distance on $D$, then the family of sets
    \begin{gather*}
      U_r = \bigl\{ \tp(a,b) \in \bG : d(a,b) < r \bigr\},
    \end{gather*}
    for $r > 0$, forms a basis of open neighbourhoods for $\bB$ in $\bG$.
  \end{enumerate}
\end{lem}
\begin{proof}
  It is easy to check that $\bG$ is a topological groupoid with basis $\bB$ and the stated source and target.
  Since the language is countable, the space $\tS_{D \times D}(T)$ is compact metrisable, and therefore Polish.
  As a condition on $\tp(a,b)$, the property $\dcl(a) = \dcl(b)$ is $G_\delta$ by \cite[Lemma~5.1]{BenYaacov:ReconstructionGroupoid}, and $a,b \notin D^0$ is open.
  Therefore $\bG$ is Polish, as the union of a closed subset and a $G_\delta$ subset of a Polish space.

  Each set $U_r$ is open and contains $\bB$.
  On the other hand, if $U$ is any open neighbourhood of $\bB$ in $\bG$, then it must be of the form $W \cap \bG$, where $W$ is an open neighbourhood of $\bB$ in $\tS_{D \times D}(T)$.
  Since $\bB$ is defined there by the condition $d(x,y) = 0$, and by compactness, $W$ must contains $\bigl[ d(x,y) < r \bigr]$ for some $r > 0$, so $U$ contains $U_r$.

  It is left to show that the target map $t\colon \bG \rightarrow \bB$ is open.
  First, consider $g \in \bG \setminus \bG^0 \subseteq \tS_{D \times D}(T)$.
  Let $[x \in D^0] \subseteq \tS_{D \times D}(T)$ be the set of types $p(x,y)$ that imply $x \in D^0$, and similarly for $y$, observing that $g \notin [x \in D^0] \cup [y \in D^0]$.
  Since this union is a closed set, $g$ admits a basis of neighbourhoods in $\tS_{D \times D}(T)$ that are disjoint from $[x \in D^0] \cup [y \in D^0]$.
  By Urysohn's Lemma and the identification of formulas with continuous functions on types, $g$ admits a basis of neighbourhoods of the form $\bigl[ \varphi(x,y) > 0 \bigr]$ where $\varphi(x,y)$ vanishes if $x \in D^0$ or $y \in D^0$.
  The family of sets $\bigl[ \varphi(x,y) > 0 \bigr] \cap \bG$ for such $\varphi$ is a basis of neighbourhoods for $g$ in $\bG$.

  Assume we are given such a neighbourhood $g \in U = \bigl[ \varphi(x,y) > 0 \bigr] \cap \bG$ (so $\varphi(x,y)$ vanishes if $x \in D^0$ or $y \in D^0$).
  Let $V = \left[ \sup_y \, \varphi(x,y) > 0 \right] \subseteq \tS_D(T) = \bB$.
  Then $V$ is open, and clearly $t(U) \subseteq V$.
  Conversely, assume that $\tp(a) \in V$, where $a \in D(M)$ for some $M \vDash T$.
  Then there exists $b \in D(M)$ such that $\varphi(a,b) > 0$.
  By hypothesis on $\varphi$, it follows that $a,b \notin D^0$.
  In particular, $a$ codes a separable $N \preceq M$, and we may assume that $b \in D(N)$.
  Now, by the density property and the uniform continuity of $\varphi$, we may assume that $b$ also codes $N$, so $\tp(a,b) \in U$.
  This proves that $t(U) = V$.

  Now let $g = \tp(a,a) \in \bG^0$.
  We have a basis of neighbourhoods of $g$ in $\bG$ consisting of sets of the form
  \begin{gather*}
    U = \bigl[ \varphi(x) > 0 \bigr] \cap \bigl[ d(x,y) < r \bigr] \cap \bG,
  \end{gather*}
  where $\varphi(a) > 0$.
  It is then easily checked that $t(U) = \bigl[ \varphi(x) > 0 \bigr]$, since we may always take $y = x$ as witness.

  This completes the proof.
\end{proof}

\begin{dfn}
  \label{dfn:UCC}
  Let $\bG$ be a topological groupoid.
  Say that a function $\varphi\colon \bG \rightarrow \bR$ is \emph{uniformly continuous and continuous (UCC)} if it is continuous on $\bG$, and in addition satisfies the following uniform continuity condition: for every $\varepsilon > 0$ there exists an open neighbourhood $U$ of the basis $\bB$ such that for every $g \in \bG$,
  \begin{gather*}
    h \in UgU \quad \Longrightarrow \quad |\varphi(g) - \varphi(h)| < \varepsilon
  \end{gather*}
\end{dfn}

Notice that unlike the situation for groups, the uniform continuity condition does not imply continuity (it is very well possible that $g_n \rightarrow h$ while $h \notin \bG g_n \bG$ for any $n$).

\begin{prp}
  \label{prp:UCC}
  Assume that $D$ is a coding sort for $T$, and let $\bG = \bG_D(T)$.
  Let $\varphi(x,y)$ be a formula on $D \times D$, and let $\varphi_\bG\colon \bG \rightarrow \bR$ be the naturally induced function
  \begin{gather*}
    g = \tp(a,b) \quad \Longrightarrow \quad \varphi_\bG(g) = \varphi(a,b).
  \end{gather*}
  Then the map $\varphi \mapsto \varphi_\bG$ defines a bijection between formulas on $D \times D$, up to equivalence, and UCC functions on $\bG$.
\end{prp}
\begin{proof}
  Let us first check that if $\varphi$ is a formula, then $\varphi_\bG$ is UCC.
  It is clearly continuous.
  The uniform continuity condition follows from the fact that $\varphi$ is uniformly continuous in each argument, together with the fact that for any $\delta > 0$ we may take choose $U = \bigl[ d(x,y) < \delta \bigr] \cap \bG$.

  Conversely, assume that $\psi \colon \bG \rightarrow \bR$ is UCC.
  By density, the function $\psi$ admits at most one continuous extension to $\tS_{D \times D}(T)$, and we need to show that one such exists.
  In other words, given $p \in \tS_{D \times D}(T)$ and $\varepsilon > 0$, it will suffice to find a neighbourhood $p \in V \subseteq \tS_{D \times D}(T)$ such that $\psi$ varies by less than $\varepsilon$ on $V \cap \bG$.
  If $p \in \bG$ this is easy, so we may assume that $p \notin \bG$.

  Let us fix $\varepsilon > 0$ first.
  By uniform continuity of $\psi$ and \autoref{lem:CodingSortGroupoid}\autoref{item:CodingSortGroupoidBasisNeighbourhoods}, there exists $\delta > 0$ such that $|\psi(g) - \psi(ugv)| < \varepsilon$ whenever $g \in \bG$, $u,v \in \bigl[d(x,y) < \delta\bigr] \cap \bG$, and $ugv$ is defined.

  Given $p = \tp(a_0,b_0)$, we may assume that $a_0,b_0 \in D(M)$ for some separable model $M$.
  Since $p \notin \bG$, we must have $a_0 \neq b_0$, and (possibly decreasing $\delta$) we may assume that $d(a_0,b_0) > 2\delta$.
  By the density property, there exist $a_1,b_1 \in D(M)$ that code $M$, with $d(a_0,a_1) + d(b_0,b_1) < \delta$, so $d(a_1,b_1) > \delta$.
  Let $g_1 = \tp(a_1,b_1) \in \bG$.
  By continuity, there exists an open neighbourhood $g_1 \in V_1 \subseteq \tS_{D \times D}(T)$ such that $|\psi(g_1) - \psi(h)| < \varepsilon$ for every $h \in V_1 \cap \bG$.
  Possibly decreasing $V_1$, we may further assume that $\tp(a,b) \in V_1$ implies $d(a,b) > \delta$
  We may even assume that $V_1$ is of the form $[\chi < \delta]$, where $\chi(x,y) \geq 0$ is a formula and $\chi(a_1,b_1) = \chi(g_1) = 0$.
  Define
  \begin{gather*}
    \chi'(x,y) = \inf_{x',y'} \, \bigl[ d(x,x') + d(y,y') + \psi(x',y') \bigr], \\
    V = \bigl[ \chi'(x,y) < \delta \bigr] \subseteq \tS_{D \times D}(T).
  \end{gather*}
  Then $V$ is open, $p \in V$, and $\tp(a,b) \in V$ implies $a \neq b$ (in other words, $V \cap \bB = \emptyset$).

  In order to conclude, consider any $g_2 = \tp(a_2,b_2) \in V \cap \bG$.
  Since $a_2 \neq b_2$, they cannot belong to the exceptional set, so both code some separable model $N$.
  By definition of $V$, there exist $a_3,b_3 \in D(N)$ such that $\chi(a_3,b_3) + d(a_2,a_3) + d(b_2,b_3) < \delta$.
  By the density property, and uniform continuity of $\chi$, we may assume that $a_3$ and $b_3$ code $N$ as well.
  Let $g_3 = \tp(a_3,b_3)$, $u = \tp(a_3,a_2)$, $v = \tp(b_2,b_3)$.
  Then $g_3 = u g_2 v \in V_1$, so
  \begin{gather*}
    |\psi(g_2) - \psi(g_1)|
    \leq |\psi(g_2) - \psi(g_3)| + |\psi(g_3) - \psi(g_1)|
    \leq 2\varepsilon.
  \end{gather*}
  Therefore $\psi$ varies by less than $4\varepsilon$ on $V \cap \bG$, which is good enough.
\end{proof}

\begin{cor}
  \label{cor:BoundedUCC}
  Every UCC function on $\bG_D(T)$ is bounded.
\end{cor}

\begin{dfn}
  \label{dfn:Norm}
  Let $\bG$ be a groupoid.
  A \emph{semi-norm} on $\bG$ is a function $\rho\colon \bG \rightarrow \bR^+$ that satisfies
  \begin{itemize}
  \item $\rho\rest_\bB = 0$, and
  \item $\rho(g) = \rho(g^{-1})$, and
  \item $\rho(gh) \leq \rho(g) + \rho(h)$, when defined.
  \end{itemize}
  It is a \emph{norm} if $\rho(g) = 0$ implies $g \in \bB$.

  A norm $\rho$ is \emph{compatible} with a topology on $\bG$ if it is continuous, and the sets
  \begin{gather*}
    \{\rho < r\} = \bigl\{g \in \bG : \rho(g) < r\bigr\},
  \end{gather*}
  for $r > 0$, form a basis of neighbourhoods for $\bB$.
\end{dfn}

\begin{cor}
  \label{cor:CorrespondenceDistanceNormUCC}
  The correspondence of \autoref{prp:UCC} restricts to a one-to-one correspondence between definable distances $d$ on $D$ and compatible norms on $\bG = \bG_D(T)$.
\end{cor}
\begin{proof}
  Let $d$ be a definable distance on $D \times D$ and $\rho_d$ the corresponding UCC function on $\bG$.
  Then $\rho_d$ is clearly a continuous norm, and it is a compatible norm by \autoref{lem:CodingSortGroupoid}\autoref{item:CodingSortGroupoidBasisNeighbourhoods}.

  The converse is more delicate.
  Let $\rho$ be a compatible norm.
  Then it is continuous, and it is easy to see that every continuous semi-norm is UCC, so $\rho = \varphi_\bG$ (in the notations of \autoref{prp:UCC}) for some formula $\varphi(x,y)$.
  If $a,b,c \in D$ all code the same separable model, then $\varphi(a,a) = 0$ and $\varphi(a,b) \leq \varphi(a,c) + \varphi(b,c)$.
  The set of types of such triplets is dense in $\tS_{D \times D \times D}(T)$, by the density property, so the same holds throughout and $\varphi$ defines a pseudo-distance.

  It is left to show that $\varphi$ defines a distance (and not merely a pseudo-distance).
  Let $d$ be any definable distance on $D$, say the one distinguished in the language.
  We already know that $\rho_d$ is a compatible norm.
  Therefore, for every $\varepsilon > 0$ there exists $\delta > 0$ such that $\{\rho < \delta\} \subseteq \{\rho_d < \varepsilon\}$.
  As in the previous paragraph, this means that the (closed) condition $\varphi(a,b) < \delta \ \Longrightarrow \ d(a,b) \leq \varepsilon$ holds on a dense set of types, and therefore throughout.
  In particular, if $\varphi(a,b) = 0$, then $a = b$, and the proof is complete.
\end{proof}

Let $T$ be a theory, $D$ a coding sort for $T$, and $\bG = \bG_D(T)$.
Then from $\bG$, given as a topological groupoid, we can essentially recover the language $\cL_D$ and the theory $T_{2D}$, as follows.
\begin{enumerate}
\item
  \label{item:ReconstructionStepNorm}
  We choose, arbitrarily, a compatible norm $\rho$ on $\bG$ (which exists, by \autoref{cor:CorrespondenceDistanceNormUCC}).
\item
  \label{item:ReconstructionStepLanguage}
  We let $\cL_\bG$ consists of a single sort, also named $D$, together with a binary predicate symbol $P_\psi$ for each UCC function $\psi$ on $\bG$.
  We know that $\psi$ is bounded (\autoref{cor:BoundedUCC}), and we impose the same bound on $P_\psi$.
  We also know that for every $\varepsilon > 0$ there exists a neighbourhood $U$ of $\bB$ such that $h \in UgU$ implies $|\psi(g) - \psi(h)| < \varepsilon$, and since $\rho$ is compatible, there exists $\delta = \delta_\psi(\varepsilon) > 0$ such that the same holds when $U = \{\rho < \delta\}$.
  We then impose the corresponding modulus of uniform continuity on $P_\psi$, namely, requiring that
  \begin{gather*}
    d(x,x') \vee d(y,y') < \delta_\psi(\varepsilon)
    \quad \Longrightarrow \quad
    |P_\psi(x,y) - P_\psi(x',y')| \leq \varepsilon.
  \end{gather*}
  We also use the bound on $\rho$ as bound on the distance predicate.
\item
  \label{item:ReconstructionModels}
  Let us fix $e \in \bB$, and consider the set
  \begin{gather*}
    e \bG = \{ g \in \bG : t_g = e \}.
  \end{gather*}
  If $g,h \in e \bG$, then $g^{-1} h$ is defined, and for any UCC $\psi$ we let:
  \begin{gather*}
    P_\psi(g,h) = \psi(g^{-1} h).
  \end{gather*}
  In particular, $d(g,h) = P_\rho(g,h) = \rho(g^{-1} h)$ is a distance function on $e \bG$.

  Assume now that $g',h' \in e \bG$ as well, and $d(g,g') \vee d(h,h') < \delta = \delta_\psi(\varepsilon)$.
  Let $u = g'^{-1} g$ and $v = h^{-1} h'$.
  Then $g'^{-1} h' = u g^{-1} h v$, and $u,v \in \{\rho < \delta\}$, so indeed
  \begin{gather*}
    |P_\psi(g,h) - P_\psi(g',h')| \leq \varepsilon,
  \end{gather*}
  as required.
  The bounds are also respected, so $e \bG$, equipped with the distance and interpretations of $P_\psi$, is an $\cL_\bG$-pre-structure, and its completion $\widehat{e \bG}$ is an $\cL_\bG$-structure.
\item
  \label{item:ReconstructionTheory}
  We define $T_\bG$ as the theory of the collection of all $\cL_\bG$-structures of this form:
  \begin{gather*}
    T_\bG = \Th_{\cL_\bG}\Bigl( \widehat{e \bG} : e \in \bB \Bigr).
  \end{gather*}
\end{enumerate}

By ``essentially recover'', we mean the following.

\begin{thm}
  \label{thm:Reconstruction}
  Let $T$ be a theory, $D$ a coding sort for $T$, and $\bG = \bG_D(T)$.
  Let $\cL_\bG$ and $T_\bG$ be constructed as in the preceding discussion.
  Then $T_\bG$ and $T_{2D}$ are one and the same, up to renaming the binary predicate symbols, and up to an arbitrary choice of the distance on the sort $D$ (from among all definable distances).

  In particular, this procedure allows us to recover from $\bG$ a theory $T_\bG$ that is bi-interpretable with $T$.
\end{thm}
\begin{proof}
  By \autoref{cor:CorrespondenceDistanceNormUCC}, step \autoref{item:ReconstructionStepNorm} consists exactly of choosing a definable distance $d$ on $D$, and the corresponding norm $\rho = d_\bG$.
  This choice is irremediably arbitrary.
  By \autoref{prp:UCC}, in step \autoref{item:ReconstructionStepLanguage} there is a natural bijection between symbols of $\cL_D$ (corresponding to formulas $\varphi(x,y)$ on $D \times D$, up to equivalence) and symbols of $\cL_\bG$: to $\varphi$ we associate the UCC function $\psi_\varphi = \varphi_\bG$, to which in turn we associate the symbol $P_{\psi_\varphi}$.

  Finally, let $M \vDash T$ be separable, let $a \in D(M)$ be a code for $M$, and let $e = \tp(a) \in \bB$.
  Let $D(M)_1$ denote the set of $b \in D(M)$ that also code $M$.
  If $b \in D(M)_1$, then $g_b = \tp(a,b) \in e \bG$.
  Moreover, if $b,c \in D(M)_1$ and $\varphi$ is a formula on $D \times D$, then $\tp(b,c) = g_b^{-1} g_c \in \bG$, so
  \begin{gather*}
    \varphi(b,c) = \psi_\varphi(g_b^{-1} g_c) = P_{\psi_\varphi}(g_b,g_c).
  \end{gather*}
  In particular, $d(b,c) = d(g_b,g_c)$ (where the first is the distance we chose on $D$, and the second the distance we defined on $e \bG$ in step \autoref{item:ReconstructionModels}).
  Thus, up to representing $\varphi$ by the symbol $P_{\psi_\varphi}$, the map $b \mapsto g_b$ defines an isomorphism of the $\cL_D$-pre-structure $D(M)_1$ with the $\cL_\bG$-pre-structure $e \bG$.
  This extends to an isomorphism of the respective completions: $D(M) \simeq \widehat{e \bG}$.

  It follows that, up to this change of language (and choice of distance), the theory $T_\bG$ defined in step \autoref{item:ReconstructionTheory} is the theory of all separable models of $T_{2D}$.
  Since $T$ is in a countable language, $T_{2D}$ is in a ``separable language'', so it is equal to the theory of all its separable models.

  By \autoref{prp:CodingSortBiInterpretation}, $T$ is bi-interpretable with $T_{2D}$, and therefore also with $T_\bG$.
\end{proof}

Having achieved this, we are ready to start producing reconstruction theorems: all we need is a coding sort that only depends (up to definable bijection) on the bi-interpretation class of $T$.

\begin{exm}
  \label{exm:ReconstructionAleph0Categorical}
  Let $T$ be an $\aleph_0$-categorical theory.
  Let $M$ be its unique separable model, and let $a$ be any sequence (possibly infinite, but countable), in any sort or sorts, such that $\dcl(a) = \dcl(M)$ (for example, any dense sequence will do).
  Let $D_{T,0}$ be the set of realisations of $p = \tp(a)$.
  Since $T$ is $\aleph_0$-categorical, $D_{T,0}$ is a definable set, i.e., a sort.
  It is easy to check that it is a coding sort (with no exceptional set).

  If $b$ is another code for $M$, and $D'_{T,0}$ is the set of realisations of $\tp(b)$, then $\dcl(a) = \dcl(b)$ and $\tp(a,b)$ defines the graph of a definable bijection $D_{T,0} \simeq D'_{T,0}$.
  Therefore, $D_{T,0}$ does not depend on the choice of $a$.
  Moreover, assume that $T'$ is an interpretational expansion of $T$.
  Then it has a model $M'$ that expands $M$ accordingly.
  But then $\dcl(M') = \dcl(M) = \dcl(a)$ (as calculated when working in $T'$), so $D_{T',0} = D_{T,0}$.
  It follows that $D_{T,0}$ only depend on the bi-interpretation class of $T$.

  Since $\tS_{D_{T,0}}(T) = \{p\}$ is a singleton, the groupoid
  \begin{gather*}
    G(T) = \bG_{D_{T,0}}(T)
  \end{gather*}
  is in fact a group.
  It only depends on the bi-interpretation class of $T$ (since $D_{T,0}$ only depends on it) and by \autoref{thm:Reconstruction}, it is a complete bi-interpretation invariant for $T$.

  We leave it to the reader to check that
  \begin{gather*}
    G(T) \simeq \Aut(M),
  \end{gather*}
  and that the reconstruction result is just a complicated restatement of those of \cite{Ahlbrandt-Ziegler:QuasiFinitelyAxiomatisable,BenYaacov-Kaichouh:Reconstruction}.
\end{exm}

\begin{exm}
  \label{exm:ReconstructionClassical}
  Let $T$ be a theory in classical logic.
  In \cite{BenYaacov:ReconstructionGroupoid}, using an arbitrary parameter $\Phi$, we gave an explicit construction of a set of infinite sequences $D_\Phi$.
  We showed that it is a definable set in the sense of continuous logic, and that its interpretation in models of $T$ only depend on the bi-interpretation class of $T$ (up to a definable bijection).
  It also follows from what we showed that it is a coding sort (without exceptional set).
  Since it is unique, let us denote it by $D_T$ (in fact, we could also just denote it by $D$: its construction only depends on the language, and then we simply restrict our consideration of it to models of $T$).
  We then proved that the groupoid
  \begin{gather*}
    \bG(T) = \bG_{D_T}(T)
  \end{gather*}
  is a complete bi-interpretation invariant for $T$.
  This is a special case of \autoref{thm:Reconstruction}.
\end{exm}

\begin{exm}
  \label{exm:ReconstructionUniversalSkloem}
  Let $T$ be a (complete) theory in continuous logic.
  In \cite{BenYaacov:ReconstructionGroupoid} we defined when a sort $D_T$ is a \emph{universal Skolem sort}, and proved that if such a sort exists, then it is unique, and only depends on the bi-interpretation class of $T$ (in contrast with the previous example, here we do not have a general construction for such a sort, let alone a uniform one, so it really does depend on $T$).
  We proved that if $T$ admits a universal Skolem sort $D_T$, then
  \begin{gather*}
    \bG(T) = \bG_{D_T}(T)
  \end{gather*}
  is a complete bi-interpretation invariant for $T$.

  Again, we also proved that $D_T$ is a coding sort, so this is a special case of \autoref{thm:Reconstruction}.
\end{exm}

\begin{rmk}
 \label{rmk:ReconstructionUniversalSkloemSpecialCases}
   \autoref{exm:ReconstructionUniversalSkloem} encompasses the two previous examples in the following sense.
  \begin{itemize}
  \item If $T$ is classical, then the sort $D_T$ of \autoref{exm:ReconstructionClassical} is a universal Skolem sort, so \autoref{exm:ReconstructionClassical} is a special case of \autoref{exm:ReconstructionUniversalSkloem}.
  \item If $T$ is $\aleph_0$-categorical, then $D_T = D_{T,0} \times 2^\bN$ is a universal Skolem sort, so
    \begin{gather*}
      \bG(T) \simeq 2^\bN \times G(T) \times 2^\bN, \qquad \text{with groupoid law} \qquad (\alpha,g,\beta) \cdot (\beta,h,\gamma) = (\alpha, gh, \gamma).
    \end{gather*}
    Consequently, $\bB(T) = 2^\bN$, and if $e \in \bB(T)$, then $G(T) \simeq e \bG(T) e$.
    Therefore, the reconstruction of \autoref{exm:ReconstructionAleph0Categorical} can be recovered from a special case of \autoref{exm:ReconstructionUniversalSkloem}.
  \end{itemize}
  In both \autoref{exm:ReconstructionClassical} and \autoref{exm:ReconstructionUniversalSkloem}, the basis $\tS_{D_T}(T)$ is homeomorphic to the Cantor space $2^\bN$.
\end{rmk}

However, in \cite{BenYaacov:ReconstructionGroupoid} we also gave an example of a continuous theory which does not admit a universal Skolem sort.
In particular, the explicit construction of $D_T$ as $D_\Phi$ in the case of a classical theory simply does not extend, as is, to continuous logic.
The rest of this article is dedicated to presenting a modified version of this construction, giving rise to a coding sort that \emph{does} have an exceptional set (a very simple one, consisting of a single point), allowing us to prove a reconstruction theorem for every first order theory in a countable language (in continuous logic, or classical one).

\section{Star spaces}
\label{sec:StarSpace}

Before we can construct our coding sort, we require technical detour, where we introduce star sets in general, and, in the model-theoretic context, star sorts.
For the time being, we must ask the reader to bear with us -- the usefulness of these notions for our goal is explained in some detail at the beginning of \autoref{sec:Witnesses}.

\begin{dfn}
  \label{dfn:StarSpace}
  A \emph{retraction set} is a set $X$ equipped with an action of the multiplicative monoid $[0,1]$.
  In particular, $1 \cdot x = x$ for all $x \in X$, and $\alpha (\beta x) = (\alpha \beta) x$ (so this is a little stronger than a homotopy).

  It is a \emph{star set} if $0 \cdot x$ does not depend on $x$.
  We then denote this common value by $0 \in X$, and call it the \emph{root} of $X$.

  A \emph{topological retraction (star) space} is one equipped with a topology making the action $[0,1] \times X \rightarrow X$ continuous.

  A \emph{metric star space} is one equipped with a distance function satisfying $d(\alpha x, \alpha y) \leq \alpha d(x,y)$ and $d(\alpha x, \beta x) = |\alpha-\beta| \|x\|$, where $\|x\| = d(x,0)$.
\end{dfn}

Notice that a retraction set $X$ can be fibred over $0 \cdot X$, with each fibre a star set.
We could also define a metric retraction space by putting infinite distance between fibres.

\begin{exm}
  \label{exm:StarSpaceLine}
  The real half line $\bR^+$ is naturally a topological and metric star space.
  The interval $[0,1]$ (or $[0,r]$ for any $r > 0$) is a compact topological and bounded metric star space.
\end{exm}

\begin{exm}
  \label{exm:StarSetProduct}
  If $X$ and $Y$ are two star sets, then $X \times Y$, equipped with the diagonal action $\alpha(x,y) = (\alpha x, \alpha y)$, is again a star set.
  If both are metric star spaces, then equipping the product with the maximum distance makes it a metric star space as well (here the maximum distance is preferable to the sum distance, since it preserves bound hypotheses on the diameter).
\end{exm}

\begin{exm}
  \label{exm:Cone}
  Let $X$ be a set, and equip $[0,1] \times X$ with the equivalence relation
  \begin{gather*}
    (\alpha, x) \sim (\beta,y) \qquad \Longleftrightarrow \qquad (\alpha,x) = (\beta,y) \quad \text{or} \quad \alpha = \beta = 0.
  \end{gather*}
  The \emph{cone} of $X$ is the quotient space
  \begin{gather*}
    *X = \bigl( [0,1] \times X \bigr) / {\sim}.
  \end{gather*}
  A member of $*X$ will be denoted $[\alpha,x]$.
  We equip it with the action $\alpha \cdot [\beta,x] = [\alpha\beta,x]$.
  This makes it a star set, with $[0,x] = 0$ regardless of $x$.

  We shall tend to identify $x \in X$ with $[1,x] \in *X$, so $[\alpha,x]$ may also be denoted by $\alpha x$.

  When $X$ is a compact Hausdorff space, the relation $\sim$ is closed, $*X$ is again compact and Hausdorff, and the identification $X \subseteq *X$ is a topological embedding.
  When $X$ is a bounded metric space, say $\diam(X) \leq 2$, we propose to metrise $*X$ by
  \begin{gather}
    \label{eq:StarDistance}
    d(\alpha x, \beta y) = |\alpha - \beta| + (\alpha \wedge \beta) d(x,y).
  \end{gather}
  In particular, if either $\alpha$ or $\beta$ vanishes, then the right hand side does not depend on either $x$ or $y$, so $d$ is well defined, and $d(0,x) = 1$ for all $x \in X$.

  The only property that is not entirely obvious is the triangle inequality, namely
  \begin{gather}
    \label{eq:StarDistanceTriangle}
    |\alpha-\gamma| + (\alpha \wedge \gamma) d(x,z)
    \leq
    |\alpha-\beta| + (\alpha \wedge \beta) d(x,y) + |\beta-\gamma| + (\beta \wedge \gamma) d(y,z).
  \end{gather}
  We may assume that $\alpha \geq \gamma$, so $\alpha \wedge \gamma = \gamma$.
  If $\beta \geq \gamma$, then \autoref{eq:StarDistanceTriangle} holds trivially since $\alpha \wedge \beta \geq \gamma = \beta \wedge \gamma$.
  If $\beta \leq \gamma$, then the right hand side evaluates to
  \begin{gather*}
    (\alpha - \gamma) + 2(\gamma - \beta) + \beta d(x,y) + \beta d(y,z).
  \end{gather*}
  Applying the triangle inequality for $X$ and the hypothesis that $2 \geq d(x,z)$, we obtain \autoref{eq:StarDistanceTriangle} in this case as well.

  We conclude that $(*X,d)$ is a metric space.
  The embedding $X \subseteq *X$ is isometric, and $\diam(*X) = 1 \vee \diam(X)$.
  If $X$ is complete, then so is $*X$.

  A special instance of this is the cone of a singleton, which can be identified with the interval $[0,1]$ equipped with the natural star, topological or metric structures.
\end{exm}

\begin{exm}
  \label{exm:GeneralisedCone}
  More generally, let $S$ be a star set, $X$ an arbitrary set, and define
  \begin{gather*}
    (s, x) \sim (t,y) \qquad \Longleftrightarrow \qquad (s,x) = (t,y) \quad \text{or} \quad s = t = 0,
    \\
    S*X = \bigl( S \times X \bigr) / {\sim}.
  \end{gather*}
  As in the definition of a cone, a member of $S*X$ will be denoted $[s,x]$ or $s * x$ (in analogy with the notation $\alpha x$).
  We make $S * X$ into a star set by defining $\alpha \cdot (s * x) = (\alpha s) * x$.

  This indeed generalises the cone construction, with $*X = [0,1] * X$.

  When $S$ and $X$ are compact Hausdorff spaces, the relation $\sim$ is closed, and $S*X$ is again compact and Hausdorff.
  When $S$ and $X$ are bounded metric spaces, say $\diam(X) \leq 2$ and $\|s\| \leq 1$ for all $s \in S$, we equip $S*X$ with the distance function
  \begin{gather*}
    d(s * x, t * y) = d(s,t) \vee d\bigl( \|s\| x, \|t\| y \bigr),
  \end{gather*}
  where $d\bigl( \|s\| x, \|t\| y \bigr)$ is calculated in $*X$.
  Notice that $\|s * x\| = \|s\|$, and the distance functions on $[0,1] * X$ and $*X$ agree.
\end{exm}

\begin{rmk}
  \label{rmk:GeneralisedConeIterated}
  The generalised cone construction of \autoref{exm:GeneralisedCone} can be easily iterated: $S * (X \times Y) = (S * X) * Y$, identifying $s* (x,y) = s * x * y$.
  In the metric case, assume that $X$ and $Y$ are both of diameter at most two.
  Equipping products with the maximum distance, $\diam(X \times Y) \leq 2$ as well, and the obvious map $*(X \times Y) \rightarrow *X \times *Y$ sending $\alpha(x,y) \mapsto (\alpha x,\alpha y)$ is isometric.
  It follows that the identification $S \times (X \times Y) = (S * X) * Y$ is isometric:
  \begin{align*}
    d(s * x * y, t * u * v)
    & = d(s * x, t * u) \vee d\bigl( \|s * x\| y, \|t * u\| v\bigr)
    \\
    & = d(s,t) \vee d\bigl( \|s\| x, \|t\| u \bigr) \vee d\bigl( \|s\|y, \|t\| v\bigr)
    \\
    & = d(s,t) \vee d\bigl( \|s\| (x,y), \|t\| (u,v) \bigr)
    \\
    & = d\bigl( s * (x,y), t * (u,v) \bigr).
  \end{align*}
  In particular, $*(X \times Y) = (*X) * Y$.
\end{rmk}

\begin{dfn}
  \label{dfn:Homogeneous}
  Let $X$ and $Y$ be two retraction (star) spaces.
  A map $f\colon X \rightarrow Y$ is \emph{homogeneous} if $f(\alpha x) = \alpha f(x)$.
  It is \emph{sub-homogeneous} if $f(\alpha x) = \beta f(x)$ for some $\beta \leq \alpha$.

  The latter will be mostly used when $Y = \bR^+$, in which sub-homogeneity becomes $f(\alpha x) \leq \alpha f(x)$.
\end{dfn}

We may also equip a retraction space with a partial order defined by $\alpha x \leq x$ whenever $\alpha \in [0,1]$.
This induces the usual partial order on $\bR^+$, and sub-homogeneity can be stated as $f(\alpha x) \leq \alpha f(x)$ for arbitrary maps between retraction spaces.
Notice also that our definition of a metric retraction space $X$ simply requires the distance function to be sub-homogeneous on $X \times X$.

\section{Star sorts}
\label{sec:StarSort}

\begin{dfn}
  \label{dfn:StarSort}
  A \emph{star sort} is a sort equipped with a definable structure of a metric star space.
  In particular, this means that the map $(\alpha,x) \mapsto \alpha x$ is definable (and not just $x \mapsto \alpha x$ for each $\alpha$).
  Star sorts will usually be denoted by $D^*$, $E^*$, and so on.
\end{dfn}

\begin{dfn}
  \label{dfn:SubHomogeneousFormula}
  Let $D^*$ be a star sort and $\varphi(u,y)$ a formula on $D^* \times E$.
  We say that $\varphi$ is \emph{sub-homogeneous} if it satisfies $\alpha \varphi(u,y) \geq \varphi(\alpha u,y) \geq 0$.

  We may specify that it is sub-homogeneous in the variable $u$, especially if $u$ is not the first variable.
  More generally, we may say that $\varphi(u,v,\ldots)$ is sub-homogeneous in $(u,v)$ if $\alpha \varphi(u,v,\ldots) \geq \varphi(\alpha u,\alpha v,\ldots) \geq 0$, and similarly for any other tuple of variables.

  If it is sub-homogeneous in the tuple of all its variables, we just say that $\varphi$ is \emph{jointly sub-homogeneous}.
\end{dfn}

\begin{exm}
  \label{exm:StarSortConstruction}
  \begin{itemize}
  \item If $D$ is any sort (of diameter at most two), then the cone $*D$, equipped with the distance proposed in \autoref{exm:Cone}, is a star sort.
    More generally, if $D^*$ is a star sort and $E$ an arbitrary sort, then $D^* * E$, as per \autoref{exm:GeneralisedCone}, is a star sort.
  \item Any finite product of star sorts, equipped with the diagonal action of $[0,1]$ and the maximum or sum distance, is again a star sort.
    Similarly, any countable product of star sorts, equipped with $d(u,v) = \sum_n \frac{d_n(u_n,v_n)}{2^n \diam(d_n)}$, is again a star sort, and the same holds with supremum in place of sum.
  \item If $D^*$ is a star sort and $d'(u,v)$ a jointly sub-homogeneous definable pseudo-distance on $D^*$, then the quotient $(D^*,d')$ can be equipped with an induced star structure, making it again a star sort.
  \item Let $D^*$ be a star sort and $E^* \subseteq D^*$ a definable subset.
    Then the distance $d(u,E^*)$ is sub-homogeneous if and only if $E^*$ is closed under multiplication by $\alpha \in [0,1]$, in which case $E^*$ is again a star sort.
  \end{itemize}
\end{exm}

Notice that $\varphi(u,y)$ is sub-homogeneous in $u$ if for every fixed parameter $b$, the formula $\varphi(u,b)$ (in $u$ alone) is sub-homogeneous.

For an alternate point of view, notice that a sub-homogeneous formula $\varphi(u,y)$ does not depend on $y$ when $u = 0$.
It can therefore be viewed as a formula $\varphi(u * y)$ in the sort $D^* * E$ (see \autoref{exm:GeneralisedCone}).
Since $\alpha (u * y) = (\alpha u) * y$, a sub-homogeneous (in $u$) formula $\varphi(u,y)$ is the same thing as a sub-homogeneous formula $\varphi(u * y)$ in a single variable from the sort $D^* * E$.

Similarly, a formula $\varphi(u,v)$ on $D^* \times E^*$ is jointly sub-homogeneous if and only if it is sub-homogeneous as a formula on the product star sort.

\begin{qst}
  We ordered the clauses of \autoref{exm:StarSortConstruction} in order to reflect the three operations by which we construct sorts in general.
  Still, something more probably needs to be said regarding the construction of sub-homogeneous pseudo-distance functions.
  In the usual context of plain sorts (and plain pseudo-distances), to every formula $\varphi(x,t)$ on $D \times E$ we can associate a formula on $D \times D$, defined by
  \begin{gather*}
    d_\varphi(x,y) = \sup_t \, |\varphi(x,t) - \varphi(y,t)|.
  \end{gather*}
  This is always a definable pseudo-distance on $D$.
  Moreover, in the case where $E = D$ and $\varphi$ already defines a pseudo-distance, $d_\varphi$ agrees with $\varphi$.

  Can something analogous be done in the present context as well?
\end{qst}

The following essentially asserts that we can retract continuously (with Lipschitz constant one, even) all formulas into sub-homogeneous ones.
The analogous result for a formula in several variables, with respect to joint sub-homogeneity in some of them, follows.

\begin{prp}
  \label{prp:SubHomogeneousFromArbitrary}
  Let $D^*$ be a star sort and $\varphi(u,y) \geq 0$ a positive formula on $D^* \times E$.
  For $k \in \bN$, define
  \begin{gather*}
    (\SH_k \varphi)(u,y) = \inf_{u',\alpha} \, \Bigl( \alpha \varphi(u',y) + k d(\alpha u',u)\Bigr), \qquad \text{where}\ u' \in D^*, \ \alpha \in [0,1].
  \end{gather*}
  \begin{enumerate}
  \item For any $\varphi \geq 0$ and $k$, the formula $(\SH_k \varphi)(u,y)$ is $k$-Lipschitz and sub-homogeneous in $u$, and $\SH_k \varphi \leq \varphi$.
  \item For any two formulas $\varphi, \psi \geq 0$ and $r \geq 0$, if $\varphi \leq \psi + r$, then $\SH_k \varphi \leq (\SH_k \psi) + r$.
    Consequently, $|(\SH_k \varphi) - (\SH_k \psi)| \leq |\varphi - \psi|$.
  \item If $\varphi$ is sub-homogeneous, then $(\SH_k \varphi) \rightarrow \varphi$ uniformly, at a rate that only depends on the bound and uniform continuity modulus of $\varphi$.
  \end{enumerate}
\end{prp}
\begin{proof}
  Clearly, $(\SH_k \varphi)(u,y)$ is $k$-Lipschitz in $u$.
  If $(\SH_k \varphi)(u,y) < r$ and $\beta \in [0,1]$, then there exist $u'$ and $\alpha$ such that $\alpha \varphi(u',y) + k d(\alpha u',u) < r$.
  Then $\alpha \beta \varphi(u',y) + k d(\alpha \beta u', \beta u) < \beta r$, showing that $(\SH_k \varphi)(\beta u) < \beta r$.
  This proves sub-homogeneity.
  We also always have $(\SH_k \varphi)(u,y) \leq 1 \cdot \varphi(u,y) + d(1 \cdot u, u) = \varphi(u,y)$.

  The second item is immediate.

  For the third item, we assume that $\varphi$ is sub-homogeneous, in which case
  \begin{gather*}
    (\SH_k \varphi)(u,y) = \inf_{u'} \, \Bigl( \varphi(u',y) + k d(u',u) \Bigr) \leq \varphi(u,y).
  \end{gather*}
  Say that $|\varphi| \leq M$ and $d(u,u') < \delta$ implies $|\varphi(u,y) - \varphi(u',y)| < \varepsilon$, and let $k > 2M / \delta$.
  If $d(u',u) \geq \delta$, then $\varphi(u',y) + k d(u',u) \geq \varphi(u)$, so such $u'$ may be ignored.
  Restricting to those where $d(u',u) < \delta$, we see that $(\SH_k \varphi) \geq \varphi - \varepsilon$.
\end{proof}

\begin{dfn}
  \label{dfn:WitnessNormalisedFormula}
  We say that a formula $\varphi(x,y)$ is \emph{witness-normalised} (in $x$, unless another variable is specified explicitly) if $\inf_y \varphi = 0$ (equivalently, if $\varphi \geq 0$ and $\sup_x \inf_y \varphi = 0$).

  More generally, for $\varepsilon > 0$, we say that $\varphi(x,y)$ is \emph{$\varepsilon$-witness-normalised} (in $x$) if $0 \leq \inf_y \varphi \leq \varepsilon$.
\end{dfn}

Witness-normalised formulas are analogous to formulas $\varphi(x,y)$ in classical logic for which $\exists y \varphi$ is valid: in either case, we require that witnesses exist.
If $\varphi(x,y)$ is any formula, then $\varphi(x,y) - \inf_z \varphi(x,z)$ is witness-normalised (we may say that it is \emph{syntactically} witness normalised), where we subtract a ``normalising'' term.

By definition, a sub-homogeneous or a witness-normalised formula is positive.
If $\varphi$ is witness-normalised in any of its arguments and $\varphi \geq \psi \geq 0$, then so is $\psi$.
This applies in particular to the formulas $\SH_k \varphi$ constructed in \autoref{prp:SubHomogeneousFromArbitrary}, assuming $\varphi$ is witness-normalised.

\begin{dfn}
  \label{dfn:StarCorrespondence}
  Let $D^*$ and $E^*$ be two star sorts.
  A \emph{star correspondence} between $D^*$ and $E^*$ is a formula $\varphi(u,v)$ on $D^* \times E^*$ that is sub-homogeneous in $(u,v)$ and witness-normalised in each of $u$ and $v$.

  Similarly, an \emph{$\varepsilon$-star correspondence} is a jointly sub-homogeneous formula that is $\varepsilon$-witness-normalised in each argument.
\end{dfn}

\begin{rmk}
  \label{rmk:EpsilonWitnessNormalised}
  If $\varphi$ is $\varepsilon$-witness-normalised (in one of its variables), then $\varphi' = \varphi \dotminus \varepsilon$ is witness-normalised (in the same), and $|\varphi - \varphi'| \leq \varepsilon$.
  If $\varphi$ is sub-homogeneous, then so is $\varphi \dotminus \varepsilon$,

  Therefore, if $\varphi$ is an $\varepsilon$-star correspondence, then $\varphi' = \varphi \dotminus \varepsilon$ is a star correspondence, and $|\varphi - \varphi'| \leq \varepsilon$.
\end{rmk}

Say that a definable map $\sigma\colon D \rightarrow E$ is \emph{densely surjective} if it is surjective in every sufficiently saturated model of the ambient theory, or equivalently, if $\sigma$ has dense image in every model.
Recall that a definable map $\sigma\colon D^* \rightarrow E^*$ between star sorts is \emph{homogeneous} if $\sigma(\alpha u) = \alpha \sigma(u)$.

Notice that a definable map $\sigma \colon D^* \rightarrow E^*$ is homogeneous if and only if the formula $d(\sigma u,v)$ is sub-homogeneous in $(u,v)$, and it is always witness-normalised in $u$.
If $\sigma$ is densely surjective, then it is homogeneous if and only if $d(\sigma u,v)$ is a star correspondence.
If $\sigma$ is bijective, then this is further equivalent to if $d(u, \sigma^{-1} v)$ being a star correspondence.

\begin{dfn}
  \label{dfn:UniversalStarSort}
  Say that a star sort $D^*$ is \emph{universal} (as a star sort) if for every star sort $E^*$, every star correspondence $\varphi$ between $D^*$ and $E^*$, and every $\varepsilon > 0$, there exists a $1/2$-star correspondence $\psi$ such that, in addition, if $\psi(u,v_i) < 1$ for $i = 0,1$, then $\varphi(u,v_i) < \varepsilon$ and $d(v_0,v_1) < \varepsilon$.
\end{dfn}

This just says that condition \autoref{item:UniversalStarSortMapsOne} of \autoref{prp:UniversalStarSortMaps}, which may be easier to parse, holds ``approximately''.
The choice of one and one half is quite arbitrary, and any two constants $0 < r_1 < r_2$ would do just as well (in the proof of \autoref{prp:UniversalStarSortMaps}\autoref{item:UniversalStarSortMapsZero} below, replace $2\psi \dotminus 1$ with $(\psi \dotminus r_1) / (r_2 - r_1)$).

\begin{prp}
  \label{prp:UniversalStarSortMaps}
  Let $D^*$ and $E^*$ be star sorts, $\varphi(u,v)$ a star correspondence on $D^* \times E^*$, and $\varepsilon > 0$.
  \begin{enumerate}
  \item
    \label{item:UniversalStarSortMapsZero}
    If $D^*$ is a universal star sort, then there exists $\psi$ as in \autoref{dfn:UniversalStarSort} that is a star correspondence (rather than a mere $\varepsilon$-star correspondence).
  \item
    \label{item:UniversalStarSortMapsOne}
    If $D^*$ is a universal star sort, then there exists a densely surjective homogeneous definable map $\sigma\colon D^* \rightarrow E^*$ such that $\varphi(u,\sigma u) \leq \varepsilon$.
  \item
    \label{item:UniversalStarSortMapsTwo}
    If both $D^*$ and $E^*$ are both universal star sorts, then the same can be achieved with $\sigma$ bijective.
  \end{enumerate}
\end{prp}
\begin{proof}
  For \autoref{item:UniversalStarSortMapsZero}, let $\psi$ be as in the conclusion of \autoref{dfn:UniversalStarSort}.
  Then $2\psi \dotminus 1$ will do.

  For \autoref{item:UniversalStarSortMapsOne}, define a sequence of formulas $\varphi_n(u,v)$ as follows.
  We start with $\varphi_0 = \varphi$, and we may assume that $0 < \varepsilon < 1$.
  Then, assuming that $\varphi_n$ is a star correspondence, we find a star correspondence $\varphi_{n+1}$ such that $\varphi_{n+1}(u,v_i) < 1$ implies $\varphi_n(u,v_i) \leq \varepsilon$ and $d(v_0,v_1) < \varepsilon/2^n$.
  Let $X_n \subseteq D^* \times E^*$ be the (type-definable) set defined by $\varphi_n \leq \varepsilon$ and $X = \bigcap X_n$.
  By hypothesis, for every $u \in D^*$ and $n$, there exists $v \in E^*$ such that $(u,v) \in X_n$.
  We also have $X_{n+1} \subseteq X_n$, so in a sufficiently saturated model there exists $v \in E^*$ such that $(u,v) \in X$.
  By the second hypothesis on $\varphi_n$, such $v$ is unique, so $X$ is the graph of a definable map $\sigma$ (and $v$ belongs to any model that contains $u$).
  By the same reasoning as above, for every $v \in E^*$ there exists $u \in D^*$ (not necessarily unique, so potentially only in a sufficiently saturated model) such that $(u,v) \in X$, so $\sigma$ is densely surjective.

  Assume now that $v = \sigma u$, i.e., $(u,v) \in X$.
  Since each $\varphi_n$ is sub-homogeneous, $(\alpha u,\alpha v) \in X$ for every $\alpha \in [0,1]$, i.e., $\alpha v = \sigma(\alpha u)$, and $\sigma$ is homogeneous.
  Finally, since $\varphi_0 = \varphi$, we have $(u,\sigma u) \in X \subseteq X_0$, so $\varphi(u, \sigma u) \leq \varepsilon$.

  For \autoref{item:UniversalStarSortMapsTwo} we use a back-and-forth version of the previous argument, with the roles of $D^*$ and $E^*$ reversed at odd steps.
\end{proof}

Notice that the zero formula is (trivially) a star correspondence on any two star sorts.
Therefore, if a universal star sort exists, then it is unique, up to a homogeneous definable bijection.

\begin{lem}
  \label{lem:UniversalStarSortLimit}
  Let $(D^*_n)$ be an inverse system of star sorts, where each $\pi_n\colon D^*_{n+1} \rightarrow D^*_n$ is surjective and homogeneous.
  \begin{enumerate}
  \item The inverse limit $D^* = \varprojlim D^*_n$ is a star sort, with the natural action $\alpha (u_n) = (\alpha u_n)$ and the distance proposed in \autoref{exm:StarSortConstruction}.
  \item A star correspondence between $D^*$ and $E^*$ that factors through $D^*_n \times E^*$ is the same thing as a star correspondence between $D^*_n$ and $E^*$.
  \item In order for $D^*$ to be a universal star sort, it is enough for it to satisfy the condition of \autoref{dfn:UniversalStarSort} for star-correspondences $\varphi$ that factor through $D^*_n \times E^*$ for some $n$.
  \end{enumerate}
\end{lem}
\begin{proof}
  The first two assertions are fairly evident.
  In what follows, we are going to identify a formula $\varphi(u_n,v)$ on $D^*_n \times E^*$ with the formula $\varphi\bigl( \pi_n(u), v \bigr)$ on $D^* \times E^*$, which is essentially what the second point says.

  For the last one, say that $\varphi$ is a star correspondence between $D^*$ and $E^*$, and let $\varepsilon > 0$.
  For $n$ large enough we may find a formula $\varphi_1(u_n,v)$ on $D^*_n \times E^*$ such that $\varphi \geq \varphi_1 \geq \varphi \dotminus \varepsilon$ (with the identification proposed in the previous paragraph).
  Since $\varphi$ is jointly sub-homogeneous, so is $\varphi \dotminus \varepsilon$.
  Using the construction of \autoref{prp:SubHomogeneousFromArbitrary}, this implies that for large enough $k$ we have
  \begin{gather*}
    \varphi \geq \SH_k \varphi \geq \SH_k \varphi_1 \geq \SH_k (\varphi\dotminus\varepsilon) \geq \varphi \dotminus 2\varepsilon.
  \end{gather*}
  Since $\varphi' = \SH_k \varphi_1$ is jointly sub-homogeneous, it a star correspondence, and it factors through $D^*_n \times E^*$.
  Assume now that $\psi(u,v)$ exists, as per \autoref{dfn:UniversalStarSort}, for $\varphi'$ and $\varepsilon$.
  In particular, if $\psi(u,v) < 1$, then $\varphi'(u,v) < \varepsilon$, so $\varphi(u,v) < 3\varepsilon$, which is good enough.
\end{proof}

\section{Sorts with witnesses}
\label{sec:Witnesses}

In this section, we provide an explicit construction of a universal star sort.
We follow a path similar to the construction of $D_\Phi$ in \cite{BenYaacov:ReconstructionGroupoid}, seeking a sort that contains ``all witnesses''.

Let us consider first the case of a single formula $\varphi(x,y)$ on $D \times E$, which we assume to be witness-normalised (namely, such that $\inf_y \varphi = 0$, see \autoref{dfn:WitnessNormalisedFormula}).
The sort $D$ is viewed as the sort of \emph{parameters}, and $E$ is the sort of potential \emph{witnesses}.
One may then wish to consider the set of ``parameters with witnesses'', namely the collection of all pairs $(x,y)$ such that $\varphi(x,y) = 0$, but this may be problematic for several reasons.

First of all, in a fixed (non-saturated) structure, for all $a$ there exist $b$ such that $\varphi(a,b)$ is arbitrarily small, but not necessarily such that $\varphi(a,b) = 0$.
This can be overcome by allowing an error, e.g., by considering all the solution set of $\varphi(x,y) \leq \varepsilon$ for some $\varepsilon > 0$.
In fact, it is enough to consider the solution set of $\varphi(x,y) \leq 1$: if we want a smaller error, we need only replace $\varphi$ with $\varphi/\varepsilon$.

A second, and more serious issue, is that the resulting set(s) need not be definable.
That is to say that it may happen that $1 < \varphi(a,b) < 1+\varepsilon$ for arbitrarily small $\varepsilon > 0$ without there existing a pair $(a',b')$ close to $(a,b)$ such that $\varphi(a',b') \leq 1$.
We can solve this by allowing a \emph{variable error}, considering triplets $(r,x,y)$ where $r \in \bR$ and $\varphi(x,y) \leq r$.
Now, if $\varphi(x,y) < r + \varepsilon$, then the triplet $(r,x,y)$ is very close to $(r+\varepsilon,x,y)$, which does belong to our set.

This may seem too easy, and raises some new issues.
For example, if we allow errors greater than the bound for $\varphi$, then the condition $\varphi(x,y) \leq r$ becomes vacuous.
This is not, in fact, a real problem, since soon enough we are going to let $\varphi$ vary (or more precisely, consider an infinite family of formulas simultaneously), and any finite bound $r$ will be meaningful for \emph{some} of the formulas under consideration.
However, in order for the previous argument to work, $r$ cannot be bounded (we must always be able to replace it with $r + \varepsilon$).
By compactness, $r = +\infty$ must be allowed as well -- and now there is no way around the fact that $\varphi(x,y) \leq \infty$ is vacuous, regardless of $\varphi$.

We seem to be chasing our own tail, each time shovelling the difficulty underneath a different rug -- indeed, a complete solution is impossible, or else we could construct a universal Skolem sort, which was shown in \cite{BenYaacov:ReconstructionGroupoid} to be impossible in general.
What we propose here is a ``second best'': allow infinite error, but use the formalism of star sorts to identify all instances with infinite error as the distinguished root element.
Thus, at the root, all information regarding the (meaningless) witnesses will be lost, while every point outside the root will involve finite error, and therefore meaningful witnesses.
Since we want the root to be at zero, rather than at infinity, we replace $r \in [1,\infty]$ with $\alpha = 1/r \in [0,1]$.

Let $D^*$ be a star sort, $E$ a sort.
The set $D^* * E = \{ u * y : u \in D^*, \, y \in E\}$, as per \autoref{exm:GeneralisedCone}, is again a star sort, in which $0 * y = 0$ regardless of $y$.

\begin{lem}
  \label{lem:OneWitness}
  Let $D^*$ be a star sort, $E$ a sort, and let $\varphi(u,y)$ a formula on $D^* \times E$, witness-normalised and sub-homogeneous in $u$.
  Then
  \begin{gather*}
    D^*_\varphi
    = \bigl\{  u * y : u \in D^* \ \text{and} \ \varphi(u,y) \leq 1 \bigr\}
    \subseteq D^* * E
  \end{gather*}
  is again a star sort, and the natural projection map $D^*_\varphi \rightarrow D^*$, sending $u * y \mapsto u$, is surjective.
\end{lem}
\begin{proof}
  We may view $\varphi$ as a formula on $D^* * E$, since, by sub-homogeneity, $\varphi(0,y) = 0$ regardless of $y$.
  The set $D^*_\varphi$ is the zero-set in $D^* * E$ of the formula $\varphi \dotminus 1$.
  Assume now that $a * b \in D^* * E$ and $\varphi(a,b) \dotminus 1 < \delta$.
  Then $(1-\delta) a * b \in D^*_\varphi$, and it is as close as desired (given $\delta$ small enough) to $a * b$.
  Therefore, $D^*_\varphi$ is definable.
  Since $\varphi$ is sub-homogeneous, $D^*_\varphi$ is closed under multiplication by $\alpha \in [0,1]$ and is therefore a star sort.
  Since $\varphi$ is witness-normalised, the projection is onto.
\end{proof}

Let us iterate this construction.
Recall from \autoref{rmk:GeneralisedConeIterated} that $(*D) * E = *(D \times E)$, identifying $(\alpha x) * y = \alpha(x,y)$.
Therefore, if $D^* \subseteq *D$ (with the induced star structure), then $D^* * E \subseteq *(D \times E)$.

\begin{dfn}
  \label{dfn:StarPhi}
  Fix a sort $D$, as well as a sequence of formulas $\Phi = (\varphi_n)$, where each $\varphi_n(x_{<n},y)$ is a witness-normalised formula on $D^n \times D$.
  Since $\Phi$ determines the sort $D$, we shall say that $\Phi$ is a sequence \emph{on} $D$.
  We then define
  \begin{gather*}
    D^*_n = \bigl\{ \alpha x_{<n} : \alpha \varphi_k(x_{<k},x_k) \leq 1 \ \text{for all} \ k < n \bigr\} \subseteq *(D^n), \\
    D^*_\Phi = \bigl\{ \alpha x : \alpha \varphi_n(x_{<n},x_n) \leq 1 \ \text{for all} \ n \bigr\} \subseteq *(D^\bN).
  \end{gather*}
\end{dfn}

In other words,
\begin{gather*}
  D^*_0 = [0,1] = *(\text{singleton}), \qquad
  D^*_{n+1} = (D^*_n)_{\varphi_n'}, \qquad
  D^*_\Phi = \varprojlim D^*_n,
\end{gather*}
where $\varphi_n'(\alpha x_{<n},y) = \alpha \varphi(x_{<n},y)$.
By \autoref{lem:OneWitness}, each $D^*_n$ is a star sort, and the natural projection $D^*_{n+1} \rightarrow D^*_n$ is onto.
By \autoref{lem:SortInverseLimit}, $D^*_\Phi = \varprojlim D^*_n$ is also a sort, and therefore a star sort by \autoref{lem:UniversalStarSortLimit}.

Notice that any formula in $D^*_n$ can be viewed, implicitly, as a formula in $D^*_k$ for any $k \geq n$, or even in $D^*_\Phi$, via the projections $D^*_k \twoheadrightarrow D^*_n$ or $D^*_\Phi \twoheadrightarrow D^*_n$ (this is, essentially, an addition of dummy variables).
In what follows, variables in $D^*_n$ will be denoted by $u_n$ or $\alpha x_{<n}$ (where $x_{<n} \in D^n$), and similarly, variables in $D^*_\Phi$ will be denoted by $u$ or $\alpha x$.

\begin{dfn}
  \label{dfn:RichSequence}
  We say that the sequence $\Phi$ on a sort $D$ is \emph{rich} if $D$ admits a definable projection onto any countable product of basic sorts, and for every witness-normalised formula $\varphi(x_{<n},y)$ in $D^n \times D$ and every $\varepsilon > 0$ there exist arbitrarily big $k \geq n$ such that $|\varphi_k(x_{<k},y) - \varphi(x_{<n},y)| < \varepsilon$ (so $\varphi$ is viewed as a formula in $x_{<k},y$ through the addition of dummy variables).
\end{dfn}

\begin{lem}
  \label{lem:RichSequenceExists}
  Under our standing hypothesis that the language is countable, with countably many basic sorts, there exists a rich sequence $\Phi$ (on an appropriate sort $D$).
  Moreover, we may construct $\Phi$ (and $D$) in a manner that only depends on the language and not on the theory of any specific structure.
\end{lem}
\begin{proof}
  For $D$ we may take the (countable) product of all infinite countable powers of the basic sorts.
  For each $k$ we may choose a countable dense family of formulas on $D^k \times D$, call them $\psi_{k,m}(x_{<k},y)$.
  Replacing them with $\chi_{k,m}(x_{<k},y) = \psi_{k,m}(x_{<k},y) - \inf_z \psi_{k,m}(x_{<k},z)$, we obtain a countable dense family of witness-normalised (in $x_{<k}$) formulas on $D^k \times D$.
  We may now construct a rich sequence $\Phi$ in which each $\chi_{k,m}$ occurs infinitely often (with additional dummy $x$ variables).
\end{proof}

Let $\Phi = (\varphi_n)$ (and $D$) be fixed, with $\Phi$ rich.
We define a formula on $D^n$ by
\begin{gather*}
  \rho_n(x_{<n}) = \frac{1}{1 \vee \bigvee_{k < n} \varphi_k(x_{<k},x_k)}.
\end{gather*}
In other words, $\rho_n(x_{<n})$ is the maximal $\alpha \in [0,1]$ such that $\alpha x_{<n} \in D^*_n$, or equivalently, such that $x_{<n}$ can be extended to $x$ with $\alpha x \in D^*_\Phi$.

\begin{lem}
  \label{lem:StarPhiCorrespondenceRho}
  Let $\Phi = (\varphi_n)$ be rich.
  Let $E^*$ be another star sort, $\psi(u_n,v)$ a star correspondence on $D^*_\Phi \times E^*$ that factors through $D^*_n \times E^*$, and $\varepsilon > 0$.
  Then $\psi$ factors through $D^*_k \times E^*$ for every $k \geq n$, and for every large enough $k$ the formula $\psi_1^k(x_{<k},v) = \psi\bigl(\rho_k(x_{<k}) x_{<n}, v \bigr)$ is $\varepsilon$-witness-normalised in either argument.
\end{lem}
\begin{proof}
  If $k \geq n$, then $\rho_k(x_{<k}) \leq \rho_n(x_{<n})$, so $\rho_k(x_{<k}) x_{<n} \in D^*_n$.
  Since $\psi(u_n,v)$ is witness-normalised in $u_n$, $\psi_1^k(x_{<k},v)$ is witness-normalised in $x_{<k}$.
  It is left to show that for $k$ large enough, it is also $\varepsilon$-witness-normalised in $v$.

  Our hypothesis regarding $D$ implies, among other things, that there exists a surjective definable map $\chi\colon D \rightarrow [0,1]$ (namely, a surjective formula).
  Therefore, for a constant $C$ that we shall choose later, there exists $m \geq n$ such that $C \chi(y) \geq \varphi_m(x_{<m},y) \geq C \chi(y) - 1/C$.

  Assume that $k > m$.
  For every possible value of $v \in E^*$, which we consider as fixed, there exists $\alpha x_{<n} \in D^*_n$ such that $\psi(\alpha x_{<n},v) < \varepsilon$.
  We can always extend $x_{<n}$ to $x_{<m}$ in such a manner that $\rho_m(x_{<m}) = \rho_n(x_{<n}) \geq \alpha$, so $\alpha x_{<m} \in D^*_m$.
  We choose $x_m$ so $\chi(x_m) = (\alpha C \vee 1)^{-1}$, and extend $x_{\leq m}$ to $x_{<k}$ so $\rho_k(x_{<k}) = \rho_{m+1}(x_{\leq m})$.

  If $\alpha C \geq 1$, then $1/\alpha \geq \varphi_m(x_{<m},x_m) \geq 1/\alpha - 1/C$, so $\alpha \leq \rho_{m+1}(x_{\leq m}) \leq \alpha (1-\alpha/C)^{-1}$.
  Having chosen $C$ large enough, $\rho_k(x_{<k}) = \rho_{m+1}(x_{\leq m})$ is as close to $\alpha$ as desired.
  If $\alpha C < 1$, then $0 \leq \alpha \leq 1/C$ and $0 < \rho_{k+1}(x_{\leq k}) \leq 1/(C - 1/C)$, so the same conclusion holds.

  Either way, having chosen $C$ large enough, $\psi_1^k(x_{<k},v)$ is as close as desired to $\psi(\alpha x_{<n},v)$, and in particular $\psi_1^k(x_{<k},v) < 2\varepsilon$, which is good enough.
\end{proof}

Given our hypothesis regarding $D$, every sort can be expressed as a definable subset of a quotient of $D$ by a pseudo-distance.
Such a quotient will be denoted $(D,\overline{d})$ (which includes an implicit step of identifying points at $\overline{d}$-distance zero).

\begin{conv}
  \label{conv:StarPhiEStar}
  From this point, and through the proof of \autoref{lem:StarPhiCorrespondenceUniversal}, we fix a star sort $E^*$.
  By the preceding remark, we may assume that $(E^*,d_{E^*}) \subseteq (D,\overline{d})$ isometrically, where $\overline{d}$ is a definable pseudo-distance on $D$ which we also fix.
  In particular, the distance on $E^*$ will also be denoted by $\overline{d}$.
  If $y \in D$, we denote its image in the quotient $(D,\overline{d})$ by $\overline{y}$.
\end{conv}

It is worthwhile to point out that if $\alpha x \in D^*_\Phi$, then for every $k \in \bN$ and $\delta > 0$,
\begin{gather}
  \label{eq:DPhiInequality}
  (\alpha \delta/2) \Bigl(\varphi_k(x_{<k},x_k) + 1 \Bigr)
  =
  (\delta/2) \Bigl(\alpha \varphi_k(x_{<k},x_k) + \alpha \Bigr)
  \leq \delta.
\end{gather}

Given $n \leq k$ and $\delta > 0$, let us define for $\alpha x \in D^*_\Phi$, $v \in E^*$ and $y \in D$:
\begin{gather*}
  \chi^n(\alpha x,y,v) = \inf_{w \in E^*} \, \Bigl[ \overline{d}\bigl( \alpha \rho_n(x_{<n})^{-1} w, v \bigr) + \alpha \overline{d}(\overline{y},w) \Bigr],
  \\
  \chi^{n,k}(\alpha x,v) = \chi^n(\alpha x, x_k,v) =  \inf_{w \in E^*} \, \Bigl[ \overline{d}\bigl( \alpha \rho_n(x_{<n})^{-1} w, v \bigr) + \alpha \overline{d}(\overline{x_k},w) \Bigr].
\end{gather*}
Let us explain this.
First of all, since $\alpha x \in D^*_\Phi$, we must have $\alpha \leq  \rho_n(x_{<n})$, so the expression $\alpha \rho_n(x_{<n})^{-1} w$ makes sense.
Also, if $\alpha = 0$, then $\chi^n(\alpha x,y,v) = \|v\|$ does not depend on $x$, so this is well defined.

Now, let $y \in D$ (possibly, $y = x_k$ for some $k \geq n$, but this will happen later).
We want $v$ to be equal to $\alpha \rho_n(x_{<n})^{-1} \overline{y}$, and in particular, we want $\overline{y}$ to belong to $E^*$.
We may not multiply by $\alpha \rho_n(x_{<n})^{-1}$ outside $E^*$, but we may quantify over $E^*$.
Therefore, we ask for $\overline{y}$ to be very close to some $w \in E^*$, and for $\alpha \rho_n(x_{<n})^{-1} w$, which always makes sense, to be close to $v$.

\begin{lem}
  \label{lem:StarPhiCorrespondenceChi}
  The formula $\chi^{n,k}(u,v)$ has the following properties:
  \begin{enumerate}
  \item
    \label{item:StarPhiCorrespondenceChiSubHomogeneous}
    It is jointly sub-homogeneous in its arguments.
  \item
    \label{item:StarPhiCorrespondenceChiFunctional}
    For every $n,\varepsilon > 0$ there exists $\delta = \delta(n,\varepsilon) > 0$ such that, if $\chi^n(u,y,v_i) \leq \delta$ for $i = 0,1$, then $\overline{d}(v_0,v_1) < \varepsilon$.
    In particular, for any $k$, if $\chi^{n,k}(u,v_i) \leq \delta$ for $i = 0,1$, then $\overline{d}(v_0,v_1) < \varepsilon$.
  \item
    \label{item:StarPhiCorrespondenceChiWitnessNormalised}
    Assuming that $\varphi_k(x_{<k},y) \geq 2 \overline{d}(\overline{y},E^*) / \delta - 1$, the formula $\chi^{n,k}(u,v)$ is $\delta$-witness-normalised in $u$.
  \end{enumerate}
\end{lem}
\begin{proof}
  Item \autoref{item:StarPhiCorrespondenceChiSubHomogeneous} is immediate (among other things, we use the fact that $\overline{d}$ is sub-homogeneous on $E^*$).

  For \autoref{item:StarPhiCorrespondenceChiFunctional}, assume that $\chi^n(\alpha x,y,v_i) = 0$.
  Then either $\alpha = 0$, in which case $v_i = 0$, or $\alpha > 0$, in which case we have $\overline{y} \in E^*$ and $v_i = \alpha \rho_n(x_{<n})^{-1} \overline{y}$.
  Either way, $v_0 = v_1$, and in particular $\overline{d}(v_0,v_1) < \varepsilon$.
  The conclusion follows by compactness.

  For \autoref{item:StarPhiCorrespondenceChiWitnessNormalised}, let $u = \alpha x \in D^*_\Phi$.
  By \autoref{eq:DPhiInequality} we have $\alpha \overline{d}(\overline{x_k},E^*) \leq \delta$.
  Choose $w \in E^*$ such that $\alpha \overline{d}(\overline{x_k},w) \leq \delta$, and let $v = \alpha \rho_n(x_{<n})^{-1} w$.
  Then $\chi^{n,k}(u,v) \leq \delta$.
\end{proof}

\begin{lem}
  \label{lem:StarPhiCorrespondenceUniversal}
  Let $\Phi = (\varphi_n)$ be rich.
  Let $E^* \subseteq (D,\overline{d})$ be a star sort, as per \autoref{conv:StarPhiEStar}, $\psi(u,v)$ a star correspondence on $D^*_\Phi \times E^*$, and $\varepsilon > 0$.
  Then there exist $n \leq k$ and $\delta > 0$ such that $\chi^{n,k}(u,v)$ is a $\delta$-star correspondence between $D^*_\Phi$ and $E^*$, and in addition, if $\chi^{n,k}(u,v_i) \leq 2\delta$ for $i = 0,1$, then $\psi(u,v_i) \leq \varepsilon$ and $\overline{d}(v_0,v_i) < \varepsilon$.
\end{lem}
\begin{proof}
  By \autoref{lem:UniversalStarSortLimit} and \autoref{lem:StarPhiCorrespondenceRho}, for some $n$ (in fact, any $n$ large enough), we may assume that $\psi$ is a star correspondence that factors as $\psi(u_n,v)$ through $D^*_n \times E^*$, and that $\psi_1(x_{<n},v) = \psi\bigl( \rho_n(x_{<n}) x_{<n}, v \bigr)$ is $\varepsilon$-witness-normalised in either argument.
  In particular, $\psi_1 \dotminus \varepsilon$ is witness-normalised.

  We may extend $\psi_1 \dotminus \varepsilon$ to $D^n \times (D,\overline{d})$, obtaining a formula $\psi_2(x_{<n},y)$ on $D^n \times D$, which is uniformly $\overline{d}$-continuous in $y$.
  Since $\psi_1 \geq 0$, we may assume that $\psi_2 \geq 0$, and even that
  \begin{gather*}
    \psi_2(x_{<n},y) \geq \overline{d}(\overline{y},E^*).
  \end{gather*}

  Let us choose $\delta > 0$ small enough, based on choices made so far.
  Since $\psi_2(x_{<n},y)$ is witness-normalised in $x_{<n}$ (choosing witnesses $\overline{y} \in E^*$), there exists $k \geq n$ such that $|\varphi_k - 2 \psi_2/\delta| \leq 1$.
  By \autoref{lem:StarPhiCorrespondenceChi}, having chosen $\delta$ small enough, the formula $\chi^{n,k}(u,v)$ is jointly sub-homogeneous, $\delta$-witness-normalised in $u$, and $\chi^{n,k}(u,v_i) \leq 2 \delta$ implies $\overline{d}(v_0,v_i) < \varepsilon$.
  There are two more properties we need to check.

  First, we need to check that $\chi^{n,k}(u,v)$ is $\delta$-witness-normalised in $v$.
  Indeed, given $v = \overline{y} \in E^*$, we know that there exists a sequence $x_{<n} \in D^n$ such that $\psi_1(x_{<n},v) \dotminus \varepsilon = 0$.
  Let $\alpha = \rho_n(x_{<n})$, so $\alpha x_{<n} \in D^*_n$, and extend the sequence $x_{<n}$ to $x_{<k}$ keeping $\alpha x_{<k} \in D^*_k$.
  We now choose $x_k = y$, so $\psi_2(x_{<n},x_k) = 0$ and $\varphi_k(x_{<k},x_k) \leq 1$.
  Therefore, $\alpha x_{\leq k} \in D^*_{k+1}$, and we may complete the sequence to $x \in D^\bN$ such that $\alpha x \in D^*_\Phi$.
  Then $\chi^{n,k}(\alpha x,v) = 0$, as witnessed by $w = v$ (recalling that we chose $\alpha = \rho_n(x_{<n})$).

  Second, we need to check that, having chosen $\delta$ appropriately, $\chi^{n,k}(\alpha x,v) \leq 2\delta$ implies $\psi(\alpha x,v) \leq \varepsilon$.
  Indeed, following a path similar to the proof of \autoref{lem:StarPhiCorrespondenceChi}\autoref{item:StarPhiCorrespondenceChiFunctional}, assume that
  \begin{gather*}
    \chi^n(\alpha x,y,v) = \alpha \psi_2(x_{<n},y) = 0.
  \end{gather*}
  If $\alpha = 0$, then $v = 0$ and $\psi(\alpha x,v) = \psi(0,0) = 0$.
  If $\alpha > 0$, then $\overline{y} \in E^*$, and $v = \alpha \rho_n(x_{<n})^{-1} \overline{y}$, and $\psi\bigl( \rho_n(x_{<n}) x,\overline{y} \bigr) \dotminus \varepsilon = \psi_2(x_{<n},y) = 0$.
  Since $(\alpha x,v) = \alpha \rho_n(x_{<n})^{-1} \bigl( \rho_n(x_{<n}) x,\overline{y} \bigr)$, it follows that $\psi(\alpha x,v) \leq \varepsilon$ in this case as well.
  By compactness, for $\delta$ small enough, if $\chi^n(\alpha x,y,v) \leq 2\delta$ and $\alpha \psi_2(x_{<n},y) \leq \delta$, then $\psi(\alpha x,v) < 2\varepsilon$.
  This last argument does not depend on $k$, so we may assume that $\delta$ was chosen small enough to begin with.
  By \autoref{eq:DPhiInequality}, the inequality $\alpha \psi_2(x_{<n},x_k) \leq \delta$ is automatic when $\alpha x \in D^*_\Phi$.
  If, in addition, we assume that $\chi^{n,k}(\alpha x,v) = \chi^n(\alpha x,x_k,v) \leq 2\delta$, then $\psi(\alpha x,v) < \varepsilon$, completing the proof.
\end{proof}

\begin{thm}
  \label{thm:RichSequenceStarSortUniversal}
  Let $\Phi$ be a rich sequence.
  Then $D^*_\Phi$ is universal.
  In particular, a universal star sort exists.
\end{thm}
\begin{proof}
  Immediate from \autoref{lem:StarPhiCorrespondenceUniversal}, using the formula $2\chi^{n,k}/\delta$.
\end{proof}

\section{Further properties of the universal star sort}
\label{sec:UniversalStarSort}

In \autoref{sec:StarSort} we showed that the universal star sort, if it exists, is unique up to a homogeneous definable bijection, and in \autoref{sec:Witnesses} we showed that one exists as $D^*_\Phi$ for any rich sequence $\Phi$.
Let us prove a few additional properties of this special sort.

\begin{conv}
  \label{conv:UniversalStarSort}
  From now on, $D^*$ denotes any universal star sort.
  Since it is unique up to a homogeneous definable bijection, multiplication by $\alpha \in [0,1]$ is well defined regardless of the construction we choose for $D^*$.
  In particular, its root is well defined.

  Notice that we can construct it as $D^*_\Phi$ in a manner that only depends on the language (and not on $T$): we obtain a universal star sort for $T$ simply by restricting our consideration of this sort to models of $T$.
\end{conv}

The uniqueness of $D^*$ means that we may choose it to be $D^*_\Phi$ for any rich $\Phi$, and in particular, that we are allowed some leverage in choosing a convenient sequence $\Phi$, as in the proof of the following result.

\begin{thm}
  \label{thm:UniversalStarSortReconstruction}
  The universal star sort $D^*$ is a coding sort for any theory $T$ (see \autoref{dfn:CodingSort}), with the exceptional set being the root $D^0 = \{0\}$.
\end{thm}
\begin{proof}
  Being a coding sort (with some exceptional set) is invariant under definable bijections (that preserve the exceptional set).
  Therefore, despite the fact that $D^*$ is only well defined up to a homogeneous definable bijection, our statement makes sense.
  We may choose a rich sequence $\Phi$ on a sort $D$, as per \autoref{dfn:RichSequence}, and take $D^* = D_\Phi^*$.

  Let $M \vDash T$ and $\alpha a \in D^*_\Phi(M) \setminus \{0\}$, and let $N = \dcl(\alpha a) \subseteq M$, necessarily a closed set (if $M$ is multi-sorted, closed in each sort separately).
  Then $\alpha \neq 0$, and $N = \dcl(a)$.
  In order to show that $N \preceq M$, it will suffice to show that it satisfies the Tarski-Vaught criterion: for every formula $\varphi(x,y)$, where $x$ is in the sort $D^\bN$ and $y$ in one of the basic sorts,
  \begin{gather*}
    \inf_y \, \varphi(a,y) = \inf_{b \in N} \, \varphi(a,b),
  \end{gather*}
  where the truth values are calculated in $M$.
  Since $D$ projects, by hypothesis, onto any basic sort, we replace $\varphi$ with its pull-back and assume that it is a formula on $D^\bN \times D$.
  Replacing $\varphi$ with $\varphi(x,y) - \inf_z \varphi(x,z)$, we may assume that $\varphi$ is witness-normalised and the left hand side vanishes.
  Then it is enough to show that for every $\varepsilon > 0$ there exists $b \in N$ such that $\varphi(a,b) < \varepsilon$, and replacing $\varphi$ with an appropriate multiple, it is enough to require $\varphi(a,b) \leq 1 + 1/\alpha$.
  Choosing $n$ such that $\varphi_n$ is a good-enough approximation of $\varphi$, it is enough to find $b \in D(N)$ such that $\varphi_n(a_{<n},b) \leq 1/\alpha$.
  For this, $b = a_n$ will do.
  This proves the coding models property of \autoref{dfn:CodingSort}.

  For the density property, assume that $M$ is separable, and let $\alpha a \in D(M)$.
  Assume first that $\alpha > 0$.
  We may freely assume that $\varphi_k = 0$ infinitely often.
  Let us fix $n_0$, and define a sequence $b \in D^\bN$ as follows.
  \begin{itemize}
  \item We start with $b_{<n_0} = a_{<n_0}$.
  \item Having chosen $b_{<k}$ (for $k \geq n_0$) such that $\alpha b_{<k} \in D^*_k$, we can always choose $b_k \in D(M)$ so $\alpha b_{\leq k} \in D^*_{k+1}$.
  \item If $\varphi_k = 0$, then we may choose any $b_k \in D(M)$ that we desire.
    Since this happens infinitely often, we may ensure that $\dcl(b) = M$.
  \end{itemize}
  In the end, $\alpha b \in D^*_\Phi$ and $\dcl(\alpha b) = \dcl(b) = M$, so $\alpha b$ codes $M$.
  Taking $n_0$ large enough, $\alpha b$ is as close as desired to $\alpha a$.

  This argument shows, in particular, that there exists $\alpha a \in D(M)$ that codes $M$.
  Let $\alpha_n = \alpha / 2^n$.
  Then $\alpha_n a \in D(M)$ codes $M$ for each $n$, and $\alpha_n a \rightarrow 0$, so the root can also be approximated by codes for $M$.
\end{proof}

\begin{dfn}
  \label{dfn:StarSortGroupoid}
  Let $T$ be any theory in a countable language, and $D^*$ its universal star sort.
  View it as a coding sort, as per \autoref{thm:UniversalStarSortReconstruction}, with exceptional set $D^0 = \{0\}$, and define the corresponding groupoid, as per \autoref{dfn:CodingSortGroupoid}:
  \begin{gather*}
    \bG^*(T) = \bG_{D^*}(T).
  \end{gather*}
\end{dfn}

We already know that this is an open Polish topological groupoid, with basis $\bB^*(T) \simeq \tS_{D^*}(T)$.

\begin{thm}
  \label{thm:StarSortGroupoid}
  The groupoid $\bG^*(T)$ is a complete bi-interpretation invariant for the class of theories in countable languages.
\end{thm}
\begin{proof}
  On the one hand, we have seen that $D^*$, and therefore $\bG^*(T)$, only depends on the bi-interpretation class of $T$.
  Conversely, by \autoref{thm:Reconstruction}, a theory bi-interpretable with $T$ (namely, the theory $T_{2D^*}$, up to some arbitrary choices of definable distance and symbols for the language) can be recovered from $\bG^*(T)$.
\end{proof}

Our last task is to calculate the basis $\tS_{D^*}(T)$ explicitly, and show how \autoref{thm:StarSortGroupoid} extends previous results, in a style similar to that of \autoref{rmk:ReconstructionUniversalSkloemSpecialCases}.

Let us fix a rich sequence $\Phi$ on a sort $D$, so we may take $D^* = D^*_\Phi$.
We also fix a formula $\chi(y)$ on $D$ that is onto $[0,1]$.
Finally, we may assume that $\varphi_n(x_{<n},y) = n \chi(y)$ for infinitely many $n$.

Let $X = \tS_{D^\bN}(T)$ and $Y = \tS_{D^*_\Phi}(T)$.
We may identify $\tS_{*D^\bN}(T)$ with $*X$, identifying $\tp(\alpha x)$ with $\alpha \tp(x)$ (here we need to assume that $T$ is complete, so there exists a unique possible complete type for $0 \in D^*_\Phi$).
This identifies $Y$ with a subset of $*X$, namely that of all $\alpha p$ where $p(x)$ implies that $\alpha x \in D^*_\Phi$, or equivalently, such that $\alpha \varphi_n(p) \leq 1$ for all $n$.

For $\alpha \in [0,1]$, let
\begin{gather*}
  X_\alpha = \{ p \in X : \alpha p \in Y \}.
\end{gather*}
In particular, $X_0 = X$.
Define $\rho\colon X \rightarrow [0,1]$ by
\begin{gather*}
  \rho(p) = \sup \, \{ \alpha : \alpha p \in Y \} = \sup \, \{ \alpha : p \in X_\alpha \}.
\end{gather*}

\begin{lem}
  \label{lem:UniversalStarSortTypes1}
  Let $\alpha > 0$.
  Then for every $p \in X$ we have $\alpha \leq \rho(p)$ if and only if $p \in X_\alpha$, and $X_\alpha$ is compact, totally disconnected.
  In particular, $\rho \colon X \rightarrow [0,1]$ is upper semi-continuous.
\end{lem}
\begin{proof}
  For the first assertion, it is enough to notice that by compactness, the supremum is attained, namely, $p \in X_{\rho(p)}$.
  It follows that the condition $\rho(p) \geq \alpha$ is equivalent to $p \in X_\alpha$, so it is closed, and $\rho$ is upper semi-continuous.

  Assume that $\alpha q_i \in Y$ and $q_0 \neq q_1$.
  Then for some finite $n$, there exists a formula $\psi(x_{<n})$ that separates $q_0$ from $q_1$, say $\psi(q_i) = i$.
  We may also find a $[0,1]$-valued formula $\chi(y)$ on $D$ that attains (at least) the values $0$ and $1$.

  By Urysohn's Lemma, there exists a formula $\varphi(x_{<n},y) \geq 0$ such that
  \begin{gather*}
    |\psi(x_{<n}) + \chi(y) - 1| \geq 1/3 \qquad \Longrightarrow \qquad \varphi(x_{<n},y) = 0, \\
    |\psi(x_{<n}) + \chi(y) - 1| \leq 1/6 \qquad \Longrightarrow \qquad \varphi(x_{<n},y) = 17/\alpha + 42.
  \end{gather*}
  Since the formula $\chi$ attains both $0$ and $1$, the formula $\varphi(x_{<n},y)$ is witness-normalised, so there exists $k \geq n$ with $|\varphi - \varphi_k| \leq 1$.

  Assume now that $\alpha p \in Y$.
  Then $\varphi_k(x_{<k},x_k)^p \leq 1/\alpha$, so $\varphi(x_{<n},x_k)^p \leq 1/\alpha +1 < 17/\alpha + 42$ and $|\psi(x_{<n}) + \chi(x_k) - 1| > 1/6$.
  This splits the set $X_\alpha$ in two (cl)open sets, defined by $\psi(x_{<n}) + \chi(x_k) > 7/6$ and $\psi(x_{<n}) + \chi(x_k) < 5/6$, respectively.
  Since $\chi$ is $[0,1]$-valued, $q_0$ must belong to the latter and $q_1$ to the former, so they can be separated in $X_\alpha$ by clopen sets, completing the proof.
\end{proof}

\begin{lem}
  \label{lem:UniversalStarSortTypes2}
  The set $X_{>0} = \bigl\{ p \in X : \rho(p) > 0 \bigr\} = \bigcup_{\alpha > 0} X_\alpha$ is totally disconnected, admitting a countable family of clopen sets $(U_n : n \in \bN)$ that separates points.
\end{lem}
\begin{proof}
  We may write $X_{>0}$ as $\bigcup_k X_{2^{-k}}$.
  Each $X_{2^{-k}}$ is compact, totally disconnected, and it is metrisable by countability of the language.
  Therefore, it admits a basis of clopen sets.

  The inclusion $X_{2^{-k}} \subseteq X_{2^{-k-1}}$ is a topological embedding of compact totally disconnected spaces.
  Therefore, if $U \subseteq X_{2^{-k}}$ is clopen, then we may find a clopen $U' \subseteq X_{2^{-k-1}}$ such that $U' \cap X_{2^{-k}} = U$.
  Proceeding in this fashion, we may find a clopen $\overline{U} \subseteq X_{>0}$ such that $\overline{U} \cap X_{2^{-k}} = U$.

  We can therefore produce a countable family of clopen sets $(U_n : n \in \bN)$ in $X_{>0}$ such that for each $k$, $\bigl( U_n \cap X_{2^{-k}} : n \in \bN \bigr)$ is a basis of clopen sets for $X_{2^{-k}}$, and in particular separates points.
  It follows that $(U_n)$ separates points in $X_{>0}$.
\end{proof}

Given this family $(U_n)$, we may define a map $\theta_0\colon X_{>0} \rightarrow 2^\bN$, where $\theta_0(p)_n = 0$ if $p \in U_n$ and $\theta_0(p)_n = 1$ otherwise.
It is continuous by definition, and injective since the sequence $(U_n)$ separates points.
If $\alpha p \in Y$, then either $\alpha = 0$ or $p \in X_{>0}$ (or possibly both), and we may define
\begin{gather*}
  \theta(\alpha p) = \alpha \theta_0(p) \in *2^\bN,
\end{gather*}
where $\theta(0) = \theta(0 \cdot p) = 0$ regardless of $p$.
It is clearly continuous at $0$, and at every point of $Y$ (since $\theta_0$ is continuous).
It is also injective on $Y$.
Since $Y$ is compact, $\theta\colon Y \rightarrow *2^\bN$ is a topological embedding.

\begin{lem}
  \label{lem:UniversalStarSortTypes3}
  The set of $\rho(p) p$ for $p \in X_{>0}$ is dense in $Y$.
\end{lem}
\begin{proof}
  We already know that $\rho(p) p \in Y$.
  Assume now that $U \subseteq Y$ is open and non-empty, so it must contain some point $\alpha p$ with $\alpha > 0$.

  We may assume that
  \begin{gather*}
    U = \Bigl\{ \beta q \in Y : |\beta - \alpha| < \varepsilon, \ q \in V \Bigr\},
  \end{gather*}
  where $V$ is an open neighbourhood of $p$ in $X$.
  The set $V$ may be taken to be defined by a condition $\psi > 0$, where $\psi(x_{<n})$ only involves finitely many variables.
  By hypothesis on $\Phi$, possibly increasing $n$, we may assume that $\varphi_n(x_{<n},y) = n \chi(y)$, and we may further assume that $\alpha > 1/n$.

  Choose a realisation $a$ of $p$.
  Let $b_{<n} = a_{<n}$ and choose $b_n$ so $\chi(b_n) = 1/n\alpha$.
  Then $\varphi_n(b_{<n},b_n) = 1/\alpha$, so $\rho_{n+1}(b_{\leq n}) = \alpha$, and we may extend $b_{\leq n}$ to a sequence $b$ such that $\rho(x') = \alpha$.
  In particular, $q = \tp(b) \in V \cap X_{>0}$ and $\alpha q = \rho(q)q \in U$.
\end{proof}

Let us recall from Charatonik \cite{Charatonik:UniqueLelekFan} a few definitions and facts regarding fans.
The \emph{Cantor fan} is the space $*2^\bN$.
It is a connected compact metrisable topological space.
More generally, a \emph{fan} $F$ is a connected compact space that embeds in the Cantor fan.
An \emph{endpoint} of $F$ is a point $x \in F$ such that $F \setminus \{x\}$ is connected (or empty, in the extremely degenerate case where $F$ is reduced to a single point).
If the set of endpoints is dense in $F$, then $F$ is a \emph{Lelek fan}.
By the main theorem of Charatonik \cite{Charatonik:UniqueLelekFan}, the Lelek fan is unique up to homeomorphism.

\begin{prp}
  \label{prp:UniversalStarSortTypesLelek}
  Let $T$ be a complete theory.
  Then $\tS_{D^*}(T)$, the type-space of the universal star sort $D^*$ in $T$, is homeomorphic to the Lelek fan.
\end{prp}
\begin{proof}
  By \autoref{lem:UniversalStarSortTypes1} to \autoref{lem:UniversalStarSortTypes3}, the space $\tS_{D^*}(T)$ is a Lelek fan.
\end{proof}

This gives us a hint as to how to relate the universal star sort with previously known coding sorts referred to in the examples of \autoref{sec:Reconstruction}.

\begin{thm}
  \label{thm:UniversalStarUninversalSkolem}
  Assume $T$ admits a universal Skolem sort $D$ in the sense of \cite{BenYaacov:ReconstructionGroupoid}, and let $L$ denote the Lelek fan.
  Then $L * D$ is a universal star sort.
\end{thm}
\begin{proof}
  We may assume that $L \subseteq *2^\bN$, and moreover, that for every non-empty open subset $U \subseteq 2^\bN$ there exists $\alpha > 0$ and $t \in U$ such that $\alpha t \in L$ (otherwise, we may replace $2^\bN$ with the intersection of all clopen subsets for which this is true).

  For each $n \in \bN$ there is a natural initial projection $2^\bN \rightarrow 2^n$.
  This induces in turn a projection $*2^\bN \rightarrow *2^n$.
  Let $L_n \subseteq *2^n$ be the image of $L$ under this projection, so $L = \varprojlim L_n$.
  Consequently, $L * D = \varprojlim \, (L_n * D)$.

  Our hypotheses regarding $L$ implies that the enpoints of $L_n$ can be enumerated as $\{ \alpha_t t : t \in 2^n \}$, with $\alpha_t > 0$.
  If $m \geq n$, then we have a natural projection $L_m \rightarrow L_n$.
  If $t \in 2^n$, $s \in 2^{m-n}$, and $ts \in 2^m$ is the concatenation, then $\alpha_{ts} ts$ gets sent to $\alpha_{ts} t \in L_n$, so $\alpha_{ts} \leq \alpha_t$, and $\alpha_{ts} = \alpha_t$ for at least one $s$.
  For any $\delta > 0$, we may always choose $m$ large enough such that for every $t \in 2^n$, the set $\{ \alpha_{ts} : s \in 2^{m-n}\}$ is $\delta$-dense in the interval $[0,\alpha_t]$.

  Let $\varphi(u,v)$ be a star correspondence between $L_n * D$ and some other star sort $E^*$, and let $\varepsilon > 0$.
  Choose $\delta > 0$ appropriately, and a corresponding $m$ as in the previous paragraph.
  Define a formula on $2^n \times 2^{m-n} \times D \times E^*$ by
  \begin{gather*}
    \varphi'(ts,x,v) = \varphi(\alpha_{ts} t * x,v).
  \end{gather*}
  On the one hand, since $\varphi$ is witness-normalised in the first argument, $\varphi'$ is witness-normalised in $(ts,x)$.
  On the other hand, if $v \in E^*$, then there exist $\alpha t \in L_n$ (so $\alpha \leq \alpha_t$) and $x \in D$ (possibly in an elementary extension) such that $\varphi(\alpha t * x,v) = 0$.
  Having chosen $\delta$ small enough to begin with, and $m$ large enough accordingly, we may now find $s \in 2^{m-n}$ such that $\alpha_{ts}$ is close to $\alpha$, sufficiently so that $\varphi'(ts,x,v) = \varphi(\alpha_{ts} t * x,v) < \varepsilon$.
  It follows that $\varphi' \dotminus \varepsilon$ is witness-normalised in either $(ts,x)$ or $v$.

  Let us now evoke a few black boxes from \cite{BenYaacov:ReconstructionGroupoid}.
  First, $2^m \times D$ is again a universal Skolem sort (and therefore stands in definable bijection with $D$).
  Second, since $\varphi' \dotminus \varepsilon$ is witness-normalised in either group of arguments, there exists a surjective definable function $\sigma\colon 2^m \times D \rightarrow E^*$ that satisfies $(\varphi' \dotminus \varepsilon) \bigl( ts,x, \sigma(ts,x)\bigr) \leq \varepsilon$, i.e., $\varphi' \bigl( ts,x, \sigma(ts,x)\bigr) \leq 2\varepsilon$.
  Define on $L_m * D \times E^*$ (keeping in mind that if $\alpha ts \in L_m$, then $\alpha \leq \alpha_{ts}$):
  \begin{gather*}
    \psi( \alpha ts * x, v) = d\bigl( v, \alpha \alpha_{ts}^{-1} \sigma(ts,x) \bigr).
  \end{gather*}
  This formula is jointly sub-homogeneous (since $d$ is, on $E^*$).
  It is also witness-normalised in $\alpha ts * x$ (just choose $v = \alpha \alpha_{ts}^{-1} \sigma(ts,x)$), and in $v$ (since $\sigma$ is surjective, and we may always choose $\alpha = \alpha_{ts}$).
  By construction, $\varphi\bigl( \alpha_{ts} t * x, \sigma(ts,x) \bigr) \leq 2\varepsilon$, so multiplying all arguments by $\alpha \alpha_{ts}^{-1}$:
  \begin{gather*}
    \varphi\bigl( \alpha t * x, \alpha \alpha_{ts}^{-1} \sigma(ts,x) \bigr) \leq 2 \varepsilon.
  \end{gather*}
  Therefore, if $\psi(\alpha ts * x, v)$ is small enough, $\varphi\bigl( \alpha t * x, v \bigr) \leq 3\varepsilon$, and by definition, if $\psi(\alpha ts * x, v_i)$ is small for $i = 0,1$, then $d(v_0,v_1)$ is small.
  Replacing $\psi$ with a multiple, we may replace ``small enough'' with ``smaller than one'', and now, by \autoref{lem:UniversalStarSortLimit}, $L * D$ is a universal star sort.
\end{proof}

\begin{cor}
  \label{cor:UniversalStarSeparablyCategorical}
  Assume that $T$ is $\aleph_0$-categorical and let $D_0$ be as in \autoref{exm:ReconstructionAleph0Categorical}.
  In other words, let $M \vDash T$ be the separable model, $a \in M^\bN$ a dense sequence, and $D_0$ the collection of realisations of $\tp(a)$.
  Then $D_0$ is a definable set, i.e., a sort, and $L * D_0$ is a universal star sort.
\end{cor}
\begin{proof}
  In an $\aleph_0$-categorical theory, every type-definable set is definable.
  By \cite[Proposition~4.17]{BenYaacov:ReconstructionGroupoid}, $2^\bN \times D_0$ is a universal Skolem sort.
  Now, $L * 2^\bN \subseteq (*2^\bN) * 2^\bN = *(2^\bN \times 2^\bN)$ is easily checked to be a fan, whose set of endpoints is dense, so it is homeomorphic to $L$.
  Therefore
  \begin{gather*}
    L * (2^\bN \times D_0) = (L * 2^\bN) * D_0 \simeq L * D_0.
  \end{gather*}
  By \autoref{thm:UniversalStarUninversalSkolem}, this is a universal star sort.
\end{proof}

Define $L^{(2)} \subseteq L^2$ as the set of pairs $(x,y)$ such that either both $x = y = 0$, or both are non-zero.
This is a Polish, albeit non-compact, star space, with root $(0,0)$.
When $\bG$ is a topological groupoid, we may equip $L^{(2)} * \bG$ with a groupoid composition law
\begin{gather*}
  [x,y,g] \cdot [y,z,h] = [x,z,gh].
\end{gather*}
If $\bB$ is the basis of $\bG$, then $L * \bB$ is the basis of $L^{(2)} * \bG$.

\begin{cor}
  \label{cor:StarGroupoidSpecialCases}
  Let $T$ be a continuous theory admitting a universal Skolem sort $D$, and let $\bG(T) = \bG_D(T)$, as in \autoref{exm:ReconstructionUniversalSkloem}.
  Then $\bG^*(T) \simeq L^{(2)} * \bG(T)$.
  If $T$ is $\aleph_0$-categorical, and $G(T)$ is the automorphism group of its unique separable model, then $\bG^*(T) \simeq L^{(2)} * G(T)$.
\end{cor}
\begin{proof}
  Just put the identities $D^* = L * D$ and $D^* = L * D_0$ through the groupoid construction.
\end{proof}

\bibliographystyle{begnac}
\bibliography{begnac}

\end{document}